\documentclass[10pt]{amsart}
\usepackage[a4paper,margin=1.00in]{geometry}
\usepackage{amsmath,amssymb,amsthm}
\usepackage{xcolor} 
\usepackage{enumitem}

\usepackage[colorlinks=true,urlcolor=blue,
citecolor=red,linkcolor=blue,linktocpage,pdfpagelabels,
bookmarksnumbered,bookmarksopen]{hyperref}

\usepackage[hyperpageref]{backref}

\usepackage{etoolbox}
\makeatletter
\patchcmd{\@settitle}
  {\uppercasenonmath\@title}
  {}{}{}
\patchcmd{\@setauthors}
  {\MakeUppercase{\authors}}
  {\authors}{}{}
\makeatother



\hypersetup{
  pdftitle={Uniqueness of prescribed-node radial bound states for a zero-mass double-power equation},
  pdfauthor={Wenjing Chen and Zexi Wang},
  pdfsubject={Prescribed-node uniqueness for radial bound states of a zero-mass double-power equation},
  pdfkeywords={zero-mass equation, radial bound state, prescribed-node uniqueness, shooting method, phase transition, inverse-branch separation}
}

\allowdisplaybreaks[2]
\numberwithin{equation}{section}
\newtheorem{theorem}{Theorem}[section]
\newtheorem{proposition}[theorem]{Proposition}
\newtheorem{lemma}[theorem]{Lemma}
\newtheorem{corollary}[theorem]{Corollary}
\theoremstyle{definition}

\newcommand{\R}{\mathbb R}
\newcommand{\dd}{\,\mathrm d}
\newcommand{\abs}[1]{\lvert#1\rvert}

\newcommand{\sgn}{\operatorname{sgn}}

\title[Prescribed-node radial bound states]{
Uniqueness of prescribed-node radial bound states for a zero-mass double-power equation}
\author{Wenjing Chen and Zexi Wang\textsuperscript{\dag}}
\date{}
\address{School of Mathematics and Statistics, Southwest University,
Chongqing 400715, China}
\email{wjchen@swu.edu.cn}\email{zxwangmath@163.com}
\thanks{\textsuperscript{\dag}Corresponding author. E-mail address:
zxwangmath@163.com.}

\thanks{The research has been supported by National Natural Science Foundation of China (No. 12671143).}


\begin{document}
\begin{abstract}
We study radial bound states of
\[
 \Delta u-\abs{u}^{q-1}u+\abs{u}^{p-1}u=0
 \quad\text{in }\R^n,
 \qquad n\ge3,\qquad 1<q<p<\frac{n+2}{n-2}.
\]
For every integer $k\ge1$, we prove that there is exactly one radial
bound state with $k$ simple zeros and $u(0)>0$. Its absolute maximum is
attained at the origin. This confirms the zero-mass, Sobolev-subcritical
case of Tang's double-power uniqueness conjecture.
The proof adapts Tang's phase-transition analysis and combines it with a finite-level separation argument inspired by Cort\'azar--Garc\'ia-Huidobro--Yarur. Three main difficulties arise in the zero-mass double-power
setting: the sign changes of key comparison quantities
in the low-amplitude region, the possible failure of
pointwise positivity of an auxiliary quantity when
$n=3$ and $p+q<3$, and the absence of exponential
decay at infinity.
We overcome these difficulties by introducing bridge
functionals to propagate the required sign information
through low amplitudes, establishing a weighted integral
positivity estimate in the exceptional three-dimensional
range, and combining a radial-flux tail estimate with
finite-level inverse-branch comparison.
We also establish the strict ordering of the bound-state shooting heights and show that the shooting variation along the \(k\)-node state has exactly \(k+1\) simple zeros.
\end{abstract}
\maketitle

\section{Introduction and main result}

We consider the semilinear scalar-field equation
\begin{equation}\label{eq:PDE}
 \Delta u+f(u)=0\quad\text{in }\R^n,
 \qquad
 f(u)=-\abs{u}^{q-1}u+\abs{u}^{p-1}u,
 \qquad 1<q<p,
\end{equation}
in dimension $n\ge3$. For a radial solution $u=u(r)$, $r=\abs{x}$,
\eqref{eq:PDE} reduces to the shooting problem
\begin{equation}\label{eq:shoot}
 u''+\frac{n-1}{r}u'+f(u)=0,
 \qquad u(0)=\alpha>0,
 \qquad u'(0)=0.
\end{equation}
A nontrivial solution of \eqref{eq:shoot} with $u(r)\to0$ as
$r\to\infty$ is called a \emph{bound state}. It is a $k$-node bound
state if it has exactly $k$ simple zeros in $(0,\infty)$. The positive
$0$-node state is the \emph{ground state}. We ask whether each nodal
class contains exactly one profile with $u(0)>0$.

The equation is in the zero-mass regime because $f'(0)=0$.
For the massive model $f(u)=-u+\abs{u}^{p-1}u$, the nondegenerate
linearization at zero gives exponential decay.
In addition,
the quantities $uf'(u)-f(u)$ and $uf(u)-2F(u)$ change sign at low
amplitudes. Both differences affect the comparison argument.

The existence theory is classical and, by now, rather complete.  Solitary
waves in higher dimensions were studied by Strauss \cite{Strauss1977}, and
the foundational work of Berestycki and Lions
\cite{BL1983I,BL1983II} established the existence of ground states and of
infinitely many radial bound states under general scalar-field assumptions.
For radial bound states, Jones and K\"upper
\cite{JonesKupper1986} used dynamical-systems methods, while McLeod, Troy and
Weissler \cite{MTW1990} gave a direct shooting theorem producing solutions
with any prescribed number of zeros.  Their hypotheses apply to the
Sobolev-subcritical double-power family considered here.  In the zero-mass
setting, Flucher and M\"uller \cite{FlucherMuller1998} established radial
symmetry and algebraic decay for ground states.  Related recent
zero-mass multiplicity results, including nodal solutions on exterior
domains, can be found in \cite{ClappCulebro2026}. 

Uniqueness is considerably subtler than existence.  For ground states,
classical shooting and separation ideas were developed by Peletier--Serrin
\cite{PeletierSerrin1983}, McLeod--Serrin \cite{McLeodSerrin1987}, and
subsequently by Chen--Lin \cite{ChenLin1991},  Kwong--Li \cite{KwongLi1992},
McLeod \cite{McLeod1993}, Erbe--Tang \cite{ErbeTang1997}, Cort\'azar--Elgueta--Felmer \cite{CEF1998}, Pucci--Serrin \cite{PS98} and Serrin--Tang
\cite{SerrinTang2000}; see also the celebrated pure-power theorem of Kwong
\cite{Kwong1989}.  For the zero-mass double-power equation, the structure of ground states was investigated by Kwong, McLeod, Peletier and
Troy \cite{KMPT1992}.
Lewin and Rota Nodari \cite{LewinRotaNodari2020} studied uniqueness and
nondegeneracy of positive solutions for a different double-power
structure, including $g(s)=s^q-s^p-\mu s$ with $p>q>1$ and $\mu>0$.
The sign of the higher-power term differs from that in \eqref{eq:PDE},
and their result does not address its nodal states. Related work on
positive double-power stationary Schr\"odinger equations appears in
\cite{SongZou2026}.
Meanwhile, uniqueness does not follow generically from standard existence hypotheses: multiple ground states arise for related nonlinearities \cite{DavilaDelPinoGuerra2013,NN85}, and multiple radial bound states with a prescribed nodal number may even exist for suitable semilinear equations \cite{CGHH2015}. Any uniqueness theorem for nodal states thus requires exploiting fine structural properties of the nonlinearity and the shooting flow.

Cort\'azar, Garc\'ia-Huidobro and Yarur \cite{CGHY2009} proved
uniqueness for a radial bound state with one zero. Their later
Theorem~1.1 in \cite{CGHY2011} treats arbitrary prescribed nodal numbers
in dimensions $2\le n\le4$ under, among other hypotheses,
\[
 \left(\frac{F}{f}\right)'(s)\ge\frac{n-2}{2},\qquad s>\beta.
\]
Their method establishes a key property for nearby
shooting heights and then uses the topology of the shooting sets.
For double powers with $p>q>0$, the cited condition holds in dimension
two; the three-dimensional range given in \cite{CGHY2011} is
$p^2+q^2\le1$. It cannot cover the present range $1<q<p$, since
\[
 \lim_{s\to\infty}\left(\frac{F}{f}\right)'(s)
 =\frac1{p+1}<\frac12\le\frac{n-2}{2},\qquad n\ge3.
\]
Their alternative double-power condition for the one-node state,
$p+q\le2/(n-2)$, is also incompatible with $1<q<p$ when $n\ge3$.

A major recent advance was made by Tang \cite{Tang2026} for the massive
pure-power equation
\begin{equation*}
 \Delta u-u+\abs{u}^{p-1}u=0,
 \qquad
 1<p<\frac{n+2}{n-2},
 \qquad n\ge3.
\end{equation*}
By analyzing the shooting variation \(v=\partial_\alpha u\), Tang developed
a phase-transition mechanism that propagates the zero/sign pattern of \(v\)
through successive nodal phases, leading to the uniqueness of the bound
state with any prescribed number of nodes.  This settled, for the massive
pure-power model, the long-standing Berestycki--Lions conjecture on uniqueness of bound states with a prescribed number of nodes; related
earlier results include \cite{AoWeiYao2016,CohenLiSchlag2024}.  More
recently, related shooting and variation arguments have been developed for
the logarithmic Schr\"odinger equation \cite{LiuSunZou2026} and for
sublinear elliptic equations \cite{ZZ26}.  Tang further conjectured
prescribed-node uniqueness for several other nonlinearities. In particular, the conjecture for the
double-power nonlinearity can be stated as follows.

\medskip
\noindent\textbf{Tang's conjecture (double-power case).}
If $f$ is given by
\begin{equation*}
 f(u)=-\abs{u}^{q-1}u+\abs{u}^{p-1}u,
 \qquad 0<q<p,
\end{equation*}
then $\Delta u+f(u)=0$ admits at most one
radial $k$-node bound state with its absolute maximum located
at the origin of $\mathbb{R}^n$, for any given $k \geq1$ and $n\geq2$.
\medskip

We prove prescribed-node uniqueness throughout the zero-mass,
Sobolev-subcritical range in dimensions $n\ge3$. Thus we confirm this
part of Tang's conjecture. 

\begin{theorem}\label{thm:main}
Assume
\begin{equation}\label{eq:main-range}
 n\ge3,
 \qquad
 1<q<p<\frac{n+2}{n-2}.
\end{equation}
Then, for every integer $k\ge1$, problem \eqref{eq:PDE} possesses exactly
one radial $k$-node bound state $u_k=u_k(r)$ satisfying $u_k(0)>0$.
\end{theorem}

Lemma~\ref{lem:nodal-geometry} shows that the absolute values of
successive extrema strictly decrease. Hence the absolute maximum of
$u_k$ is attained at the origin. 
Corollary~\ref{cor:shooting-structure} records the strict ordering of their
shooting heights and the fact that the shooting variation has exactly \(k+1\) simple zeros.


The proof relies on three key ingredients. First, the bridge functionals $B_0$ and $B_a$, whose derivatives are governed by the same comparison quantity $Q_n$, propagate the required sign information through the low-amplitude region $\abs{u}<\beta$. Second, in the three-dimensional residual range $p+q<3$, $Q_3$ need not be pointwise positive. We overcome this difficulty by establishing a weighted integral estimate that retains the factor $u^{q-1}$. A quantitative separation bound involving $\beta^{q-1}$ and $\kappa^2$, together with the parameter estimate \eqref{eq:kappa-weight-bound}, completes the finite-phase induction.

Finally, the radial flux provides the terminal estimate $r\abs{u'(r)}\to0$, while the bridge identity establishes transversality at $\abs{u}=1$. At this fixed level, the inverse radius decreases and the inverse-branch quantity $W$ increases with the shooting height, as follows from \eqref{eq:Ralpha} and \eqref{eq:Walpha}. We then propagate this strict ordering to the terminal endpoint for every $n\ge3$, obtaining the local shooting separation property and completing the global uniqueness argument without requiring the shooting variation to diverge at infinity.



The two-dimensional case follows from Cort\'azar--Garc\'ia-Huidobro--Yarur \cite{CGHY2011}, while the endpoint $q=1$ in dimensions $n\ge3$ is covered by Tang's theorem \cite{Tang2026}. Combining these results with Theorem~\ref{thm:main}, we obtain the following unified statement.

\begin{corollary}
Let $n\ge2$ and assume that
\[
\begin{cases}
1\le q<p<\infty, \qquad& n=2,\\
1\le q<p<\dfrac{n+2}{n-2}, \qquad& n\ge3.
\end{cases}
\]
Then, for every integer $k\ge1$, the equation
\[
\Delta u-\abs{u}^{q-1}u+\abs{u}^{p-1}u=0
\quad\text{in }\R^n
\]
possesses exactly one radial $k$-node bound state satisfying $u(0)>0$.
\end{corollary}

\begin{proof}
Existence follows from the prescribed-node existence theorem of \cite{MTW1990}. For $n=2$, uniqueness follows from \cite[Theorem~1.1]{CGHY2011}. For $n\ge3$, the cases $q>1$ and $q=1$ follow from Theorem~\ref{thm:main} and \cite{Tang2026}, respectively.
\end{proof}

The paper is organized as follows. Section~\ref{sec2} develops the shooting framework and the key comparison identities. Section~\ref{sec:first} analyzes the first nodal phase and introduces the bridge functionals used in the low-amplitude region. Section~\ref{sec:transition} establishes the finite-phase transition, including a weighted residual estimate for the delicate case $n=3$ and $p+q<3$. Section~\ref{sec:tail}
studies the algebraic tail and proves terminal transversality.
Finally, Section~\ref{sec:global} establishes the finite-level separation property and completes the uniqueness proof.

\section{Shooting preliminaries and comparison identities}\label{sec2}

Throughout this paper, $n\ge3$ and $1<q<p<(n+2)/(n-2)$.
Put
\[
 d:=p-q>0.
\]
For $s>0$,
\begin{equation*}
 f(s)=s^p-s^q=s^q(s^d-1),
 \qquad
 F(s):=\int_0^s f(t)\dd t
 =\frac{s^{p+1}}{p+1}-\frac{s^{q+1}}{q+1}.
\end{equation*}
The positive zero of $f$ is $1$ and the positive zero of $F$ is
\begin{equation*}
 \beta:=\left(\frac{p+1}{q+1}\right)^{\frac{1}{d}}>1.
\end{equation*}
The energy
\begin{equation*}
 E(r):=\frac12 u'(r)^2+F(u(r))
\end{equation*}
satisfies
\begin{equation}\label{eq:Eprime}
 E'(r)=-\frac{n-1}{r}u'(r)^2\le0.
\end{equation}
For a nonconstant solution, $E(r)$ decreases strictly in $(0,\infty)$.  Since $F(s)\to+\infty$ as
$|s|\to\infty$ and $F$ is bounded below, the estimate
$E(r)\le E(0)=F(\alpha)$ bounds both $u$ and $u'$ on every interval of
existence.

The standard continuation theorem therefore shows that every
shooting solution of \eqref{eq:shoot} is defined for all $r\ge0$.
Every zero of a nontrivial solution is simple, because $f$ is locally Lipschitz and a double zero would imply the zero solution by ODE uniqueness.  We record the phase geometry used below.

\begin{lemma}\label{lem:nodal-geometry}
Let $u$ be a nontrivial solution of \eqref{eq:shoot}.
\begin{enumerate}[label=\textup{(\roman*)}]
\item If $z_1$ is the first zero of $u$, then $\alpha=u(0)>\beta$ and
$u'(r)<0$ for $0<r< z_1$.
\item Between two consecutive zeros of $u$, there is exactly one critical point.
\item If $u$ is a bound state, then after its last zero, there is exactly one critical point; beyond it $u$ is monotone and converges to zero.
\item For a bound state, $E(r)>0$ for all $r>0$.
If
$c_-<c_+$ are consecutive critical points separated by a zero, then
$\abs{u(c_-)}>\abs{u(c_+)}>\beta$.
\end{enumerate}
\end{lemma}

\begin{proof}
Since $z_1$ is simple,
\[
 F(\alpha)=E(0)>E(z_1)=\frac12u'(z_1)^2>0.
\]
Thus $\alpha>\beta$, and radial regularity gives
$u''(0)=-f(\alpha)/n<0$.  Suppose that $u'$ has a first zero $c\in(0,z_1)$.
Then $u'<0$ on $(0,c)$, $u'(c)=0$ and $u''(c)\ge0$.  On the other hand,
$F(u(c))=E(c)>E(z_1)>0$, so $u(c)>\beta$ and
$u''(c)=-f(u(c))<0$.  This contradiction proves \textup{(i)}.

Rolle's theorem gives at least one critical point between consecutive zeros.  At any critical point $c$ preceding a later zero $z$, strict energy decay gives
\[
 F(u(c))=E(c)>E(z)=\frac12u'(z)^2>0,
\]
so $\abs{u(c)}>\beta$. By $u''(c)=-f(u(c))$, a critical point with $u>0$ is therefore a strict local maximum, while one with $u<0$ is a strict local minimum.  Two critical points of the same type would force a critical point of the opposite type between them, which is impossible.  This proves \textup{(ii)}.

After the last zero of a bound state, $u$ has fixed sign and tends to zero.  Since $F(u(r))\to0$ and $E$ decreases, $E(r)$ has a limit $L\ge0$.  If $L>0$, then $u'(r)^2=2(E(r)-F(u(r)))$ is bounded away from zero for all large $r$, contradicting $u(r)\to0$.  Hence $L=0$ and $u'(r)\to0$.  The derivative immediately after the last zero points into the fixed-sign half-line, so it must vanish before the solution can return to zero at infinity.  At such a critical point $E(c)>0$, hence $\abs{u(c)}>\beta$, and the same critical-point argument gives uniqueness.  This proves \textup{(iii)}.

Since $E$ is strictly
decreasing and $E(r)\downarrow0$ along the tail,
we obtain $E(r)>0$ for all $r>0$.  Finally, $E(c_+)<E(c_-)$ and $F$ is
strictly increasing as a function of $\abs{s}$ on $\abs{s}>1$.  Since both
amplitudes exceed $\beta>1$, statement \textup{(iv)} follows.
\end{proof}

\subsection{Pohozaev functionals}
Define
\begin{align*}
 h(s)&:=2nF(s)-(n-2)s f(s),\\
 h_1(s)&:=s f(s)-2F(s),\\
 h_2(s)&:=(n+2)f(s)-(n-2)s f'(s).
\end{align*}
The function $f$ is odd, while $F$, $h$, and $h_1$ are even and $h_2$ is
odd.  These parity properties will be used when consecutive phases are
normalized by reflection.
For $s>0$,
\begin{align}
 h(s)
 &=\frac{n+2-(n-2)p}{p+1}s^{p+1}
 -\frac{n+2-(n-2)q}{q+1}s^{q+1},\label{eq:h-exp}\\
 h_1(s)
 &=\frac{p-1}{p+1}s^{p+1}
 -\frac{q-1}{q+1}s^{q+1},\label{eq:h1-exp}\\
 h_2(s)
 &=\bigl[n+2-(n-2)p\bigr]s^p
 -\bigl[n+2-(n-2)q\bigr]s^q.\label{eq:h2-exp}
\end{align}
Under $q<p<(n+2)/(n-2)$, $h$ and $h_2$ each have exactly one positive zero.  This follows from the positivity of the coefficients in
\eqref{eq:h-exp}--\eqref{eq:h2-exp}.  In particular, the Pohozaev function
\begin{equation*}
 P(r):=r^n\bigl(u'^2+2F(u)\bigr)+(n-2)r^{n-1}uu'=2r^nE(r)+(n-2)r^{n-1}uu'
\end{equation*}
satisfies $P'(r)=r^{n-1}h(u(r))$.
Also set
\begin{equation*}
 P_1:=P+r^n h_1(u),
 \qquad
 P_2:=P-\frac{r^n}{n}h(u).
\end{equation*}
Then
\begin{equation*}
 \omega:=-\frac{ru'}u,
 \qquad
 \omega'=\frac{P_1}{r^{n-1}u^2},
 \qquad
 \left(\frac{P}{r^n}\right)'=-\frac{nP_2}{r^{n+1}},
\end{equation*}
where the identities involving $\omega$ are understood at positive radii at which $u\ne0$.  Moreover,
\begin{equation*}
 P_2'(r)=-\frac{r^n u'}{n}h_2(u).
\end{equation*}

\begin{lemma}\label{lem:Ppositive}
Let $u$ be a bound state of \eqref{eq:shoot}.
Then $P(r)>0$ and $P_2(r)>0$ at every positive radius belonging to a finite nodal phase of $u$.
Consequently, $P/r^n$ is strictly decreasing there.  On every subinterval where $\abs{u}\geq1$, one also has $P_1>0$ and hence $\omega'>0$.
\end{lemma}

\begin{proof}
At any simple zero $z$ of $u$,
\[
 P(z)=P_2(z)=z^n u'(z)^2>0.
\]
At any critical point \(c>0\) such that \(u\) has at least one zero in $(c,\infty)$,
Lemma~\ref{lem:nodal-geometry} gives $|u(c)|>\beta$, and therefore
\[
 P(c)=2c^nF(u(c))>0,
 \qquad
 P_2(c)=\frac{n-2}{n}c^n u(c)f(u(c))>0.
\]
At the origin, $P(0)=P_2(0)=0$.

For  $t>0$, both functions
$h(t)$ and $h_2(t)$ have exactly one positive zero, are negative
below that zero, and positive above it.  Consider first a prezero
half-phase, on which $|u|$ decreases strictly.  The identities
\[
 P'=r^{n-1}h(|u|),\qquad
 P_2'=\frac{r^n|u'|}{n}h_2(|u|)
\]
show that both $P$ and $P_2$ can first increase and then
decrease, but cannot have an interior minimum.
Since they are nonnegative at the preceding endpoint
and strictly positive at the subsequent zero,
both remain strictly positive throughout the interior.

On a postzero half-phase, $uu'>0$, and the energy is strictly
positive by Lemma~\ref{lem:nodal-geometry}.  Consequently,
\[
 P=2r^nE+(n-2)r^{n-1}uu'>0.
\]
Moreover, $|u|$ is now increasing and the parity of $h_2$ yields
\[
 P_2'=-\frac{r^n|u'|}{n}h_2(|u|).
\]
Its derivative has the increase--decrease sign pattern.  Since
$P_2>0$ at both endpoints, it is positive throughout.  This proves
$P,P_2>0$ on every finite phase.  The identity for $(P/r^n)'$ now
implies its strict decrease.  Finally $h_1(u)>0$ for $|u|\ge1$,
so $P_1=P+r^nh_1(u)>0$ and $\omega'>0$ wherever $|u|\ge 1$.
\end{proof}

The next lemma replaces the use of exponential decay in Tang's proof when positivity of $P$ is needed on the semi-tail.

\begin{lemma}\label{lem:Ptail}
Let $u$ be a bound state of \eqref{eq:shoot} and let $c$ be its last critical point.
If $c>0$, then $P(r)>0$ for $r\ge c$.  If $c=0$, then
$P(r)>0$ for $r>0$.
\end{lemma}

\begin{proof}
After reflection if necessary, assume $u>0$, $u'<0$, and $u\to0$ on
$(c,\infty)$.  Then $E(r)\downarrow0$ and $u'(r)\to0$.
For all sufficiently large $r$, one has $0<u(r)<1$ and $f(u)<0$.  Put
\[
 y(r):=-r^{n-1}u'(r)>0.
\]
Since $y'(r)=r^{n-1}f(u(r))<0$ on the small-amplitude tail, $y$ is bounded there.  From
\[
 P(r)=2r^nE(r)-(n-2)u(r)y(r)
\]
and $E(r)>0$, we obtain $P(r)>-(n-2)u(r)y(r)$,
so $\liminf_{r\to\infty}P(r)\ge0$.  On the other hand, $h(u)<0$ for all
sufficiently small $u$.  Hence $P'<0$ eventually.  The function $P$ cannot
become nonpositive on that final interval.  Indeed, if $P(r_0)\le0$ at a
sufficiently large point, then strict decrease gives
$P(r)<P(r_1)< P(r_0)\leq 0$ for every $r>r_1>r_0$.  This contradicts the preceding
lower limit.  Thus $P(r)>0$ for all sufficiently large $r$.

Along the whole monotone tail, $h(u(r))$ changes sign at most once, from
positive to negative.  
Hence the same increase--decrease argument as in Lemma~\ref{lem:Ppositive} shows that \(P\) remains positive on the whole monotone tail. This proves the claim.
\end{proof}

\subsection{Zero-mass sign structure}
For $s>0$, set
\begin{equation*}
 \mathcal A(s):=sf'(s)-f(s)
 =(p-1)s^p-(q-1)s^q.
\end{equation*}
Thus $\mathcal A$ changes sign at
\[
 \delta_A:=\left(\frac{q-1}{p-1}\right)^{\frac{1}{d}}<1.
\]
Likewise, $h_1$ changes sign at
\[
 \delta_h:=
 \left(\frac{(q-1)(p+1)}{(p-1)(q+1)}\right)^{\frac{1}{d}},
 \qquad 0<\delta_A<\delta_h<1.
\]
In particular, Tang's global conditions $-f'(0)>0$ and $h_1\ge0$ fail, because $f'(0)=0$,
but on the entire high-amplitude region $\abs{u}\geq1$, one has
\begin{equation}\label{eq:high-signs}
 \sgn \mathcal A(u)=\sgn u,
 \qquad h_1(u)>0,
 \qquad f'(u)>0,
 \qquad \sgn f''(u)=\sgn u.
\end{equation}

\subsection{High-amplitude comparison ratios}
For $s\ge1$, define
\begin{equation*}
 g_1(s):=\frac{2f(s)}{sf'(s)-f(s)}
 =\frac{2(s^d-1)}{(p-1)s^d-(q-1)}.
\end{equation*}
Then
\begin{equation*}
 g_1(1)=0,
 \qquad g_1'(s)>0,
 \qquad 0<g_1(s)<\frac{2}{p-1},\qquad s>1.
\end{equation*}
Since the zero $\delta_h$ of $h_1$ lies strictly below $1$, the quotient
\begin{equation*}
 g_2(s):=\frac{2F(s)}{h_1(s)}
\end{equation*}
is smooth for every $s\geq1$.  Writing $x=s^d$ gives the explicit formula
\begin{equation*}
 g_2(s)=
 \frac{2\bigl((q+1)x-(p+1)\bigr)}
 {(p-1)(q+1)x-(q-1)(p+1)}.
\end{equation*}
Consequently
\begin{equation*}
\begin{aligned}
 g_2'(s)>0 \quad \text{if }s>1,\qquad
  g_2(s)<0 \quad \text{if }1\leq s<\beta,\qquad g_2(\beta)=0,
 \qquad 0<g_2(s)<\frac{2}{p-1} \quad \text{if }s>\beta.
\end{aligned}
\end{equation*}
We extend $g_1, g_2$ evenly to $\abs{s}\geq1$.

A useful feature of the double-power model is the strict comparison
\begin{equation*}
 g_1(s)>g_2(s),\qquad s\geq 1.
\end{equation*}
Indeed,
\[
 g_1(s)-g_2(s)
 =\frac{2s^{p+q+1}(p-q)^2}
 {(p+1)(q+1)\mathcal A(s) h_1(s)}>0,
 \qquad s\geq1.
\]
The monotonicity behind the inflection argument is equally explicit.  Put
\begin{equation}\label{eq:varpi}
 \varpi(s):=\frac{f(s)}{s f'(s)}
 =\frac{s^d-1}{p s^d-q},\qquad s\geq1.
\end{equation}
Then $\varpi'(s)>0$ for $s>1$.

\begin{lemma}\label{lem:inflection}
Let $u$ be a monotone prezero half-phase, normalized so that $u>1$ and $u'<0$ between its initial critical point and the point where $u=1$. Assume $P_1>0$ on this interval. Then $u''$ vanishes exactly once there.  Before that point $u''<0$, and afterward $u''>0$.  In particular, $\abs{u'}$ is strictly decreasing after the inflection point.
\end{lemma}

\begin{proof}
At a positive-radius initial critical point, $u''=-f(u)<0$.  If the initial point is the origin, radial regularity gives $u''(0)=-f(\alpha)/n<0$.  At the level $u=1$, the equation gives $u''=-(n-1)u'/r>0$.  Hence a zero of $u''$ exists.  Let $r_0$ be the first one.  Then $u'''(r_0)\ge0$.  Differentiating the equation and using $u''(r_0)=0$ gives
\[
 u'''(r_0)=u'(r_0)
 \left(\frac{n-1}{r_0^2}-f'(u(r_0))\right),
\]
so, because $u'<0$,
\[
 \frac{n-1}{r_0^2}\le f'(u(r_0)).
\]
The inequality $u'''(r_0)\ge0$ is in fact strict. To see this, suppose
that $u'''(r_0)=0$ and put $u_0=u(r_0)>1$. Then
\[
 f'(u_0)=\frac{n-1}{r_0^2},\qquad
 f(u_0)=-\frac{n-1}{r_0}u'(r_0).
\]
Since $f$ is smooth near $u_0$, we may differentiate again. Since
\[
s\varpi'(s)
=1-\varpi(s)-\frac{f(s)f''(s)}{f'(s)^2},
\]
the strict positivity of $\varpi'(u_0)$ yields
\[
\frac{f(u_0)f''(u_0)}{f'(u_0)^2}
<1-\varpi(u_0)<1.
\]
Differentiating the equation twice and using
$u''(r_0)=u'''(r_0)=0$, we obtain
\[
u^{(4)}(r_0)
=-\frac{2(n-1)}{r_0^3}u'(r_0)
-f''(u_0)u'(r_0)^2=-\frac{u'(r_0)f'(u_0)}{r_0}
\left(
2-\frac{f(u_0)f''(u_0)}{f'(u_0)^2}
\right)>0.
\]
Taylor's expansion yields
\[
u''(r)=\frac12u^{(4)}(r_0)(r-r_0)^2
+o((r-r_0)^2)>0
\]
for $r<r_0$ sufficiently close to $r_0$. This contradicts the negativity of $u''$ to the left of its first zero. Hence $u'''(r_0)>0$,
and $u''$ crosses from negative to positive at $r_0$.

The equation at $r_0$ also yields
\[
 \omega(r_0)=-\frac{r_0u'(r_0)}{u(r_0)}
 =\frac{r_0^2 f(u(r_0))}{(n-1)u(r_0)}
 \geq \varpi(u(r_0)).
\]
Now $P_1>0$ implies that $\omega$ increases with $r$, while \eqref{eq:varpi} and the decrease of $u$ imply that $\varpi(u(r))$ decreases.  Thus
\[
 \omega(r)>\omega(r_0) \geq \varpi(u(r_0)) >\varpi(u(r)),\qquad r>r_0.
\]
At any later zero of $u''$, we would have
\[
 u'''=u'f'(u)\left(\frac{\varpi(u)}{\omega}-1\right)>0.
\]
If this were the first later zero, then $u''>0$ immediately to its left and hence $u'''\le0$ there.  This contradicts the displayed strict positivity.  Thus a second zero is impossible.  The final assertion follows from $u'<0$ and $u''>0$ after $r_0$.
\end{proof}

\subsection{Variation and connection identities}
Since $q>1$, the nonlinearity $f$ belongs to $C^1(\R)$.  The shooting
family depends $C^1$ on its initial height on every compact radial interval.
The formal singularity at $r=0$ causes no difficulty: near the origin the
regular radial solution is equivalently written as
\[
 u(r,\alpha)=\alpha-
 \int_0^r t^{1-n}\int_0^t s^{n-1}f(u(s,\alpha))\dd s\dd t,
\]
and differentiation of this Volterra equation with respect to $\alpha$
yields the usual parameter dependence; away from the origin it is continued
by standard ODE theory.  Thus
\[
 v(r,\alpha):=\partial_\alpha u(r,\alpha)
\]
is well defined and satisfies
\begin{equation}\label{eq:var}
 v''+\frac{n-1}{r}v'+f'(u)v=0,
 \qquad v(0)=1,
 \qquad v'(0)=0.
\end{equation}
This is the same standard shooting variation used in \cite[Section~3]{Tang2026}.
If $v(r_0)=v'(r_0)=0$ at some $r_0>0$, uniqueness for the regular linear
ODE on $(0,\infty)$ would give $v\equiv0$, contradicting $v(0)=1$.
Hence every positive zero of $v$ is simple.

Following the comparison construction in
\cite[Sections~4--5]{Tang2026}, define
\begin{align}
 M(r)&:=r^{n-1}(u'v-uv'),\label{eq:M}\\
 Q(r)&:=r^n\bigl(u'v'+f(u)v\bigr)+(n-2)r^{n-1}u'v,\label{eq:Q}\\
 Q_j(r)&:=Q(r)+j r^{n-1}u'v.\notag
\end{align}
The identities
\begin{align}
 M'&=r^{n-1}\bigl(uf'(u)-f(u)\bigr)v,\label{eq:Mprime}\\
 Q'&=2r^{n-1}f(u)v,\label{eq:Qprime}\\
 Q_1'&=r^{n-1}\bigl(u'v'+f(u)v\bigr),\notag\\
 Q_2'&=2r^{n-1}u'v',\notag\\
 Q_j'&=r^{n-1}\bigl(ju'v'-(j-2)f(u)v\bigr),\notag\\
 \left(\frac{v}{u'}\right)'&=\frac{Q_1}{r^n u'^2},\label{eq:vuratio}\\
 \left(\frac{Q}{r^{n-1}u'}\right)'&=
 f(u)\frac{Q_2}{r^{n-1}u'^2}\label{eq:Qratio}
\end{align}
are algebraic consequences of \eqref{eq:shoot} and \eqref{eq:var} and do not depend on the special form of $f$.  The last two identities are used only where $u'\ne0$.

The Wronskian of $u'$ and $v'$ is
\begin{equation}\label{eq:rho}
 \rho(r):=r^{n-1}\bigl(u''v'-u'v''\bigr)
 =r^{n-1}\bigl(f'(u)u'v-f(u)v'\bigr),
\end{equation}
with
\begin{equation}\label{eq:rhoprime}
 \rho'(r)=r^{n-1}f''(u)u'^2v.
\end{equation}
Although $f$ need not be $C^2$ at the origin when
$1<q<2$, the identity \eqref{eq:rhoprime} remains
valid wherever $u\ne0$. Moreover, $\rho$ is
continuous across nodal zeros by \eqref{eq:rho},
since $f\in C^1(\mathbb R)$ and $f(0)=f'(0)=0$.

Define
\begin{equation*}
 B_0(r):=Q(r)-2F(u(r))\frac{r^{n-1}v(r)}{u'(r)}\quad \text{if }u'\ne0.
\end{equation*}
Set
\begin{equation*}
 \Xi(r):=\frac{r^{n-1}v(r)}{u'(r)}.
\end{equation*}
Using the equation for $u$, we obtain
\[
 \Xi'=\frac{Q_n}{r u'^2}.
\]
Since $Q'=2r^{n-1}f(u)v$ and $u'\Xi=r^{n-1}v$, it follows that
\begin{equation}\label{eq:B0prime}
 B_0'
 =Q'-2f(u)u'\Xi-2F(u)\Xi'=-2F(u)\frac{Q_n}{r u'^2}.
\end{equation}

For a constant $a$, put
\begin{equation*}
 F_a(s):=F(s)-\frac a2 h_1(s),
 \qquad
 B_a(r):=Q(r)-aM(r)-2F_a(u(r))\frac{r^{n-1}v(r)}{u'(r)}.
\end{equation*}
Here and in the derivative identity below, $B_a$ is used only where
$u'\ne0$.  For the verification, write
\[
 \theta(r):=r^{n-1}\frac{h_1(u(r))v(r)}{u'(r)}=h_1(u)\Xi.
\]
Since $h_1'(s)=sf'(s)-f(s)$, identities \eqref{eq:Mprime} and the formula
for $\Xi'$ give
\begin{align*}
 \theta'
 =h_1'(u)u'\Xi+h_1(u)\Xi'
 =r^{n-1}\bigl(uf'(u)-f(u)\bigr)v
   +h_1(u)\frac{Q_n}{r u'^2}
 =M'+h_1(u)\frac{Q_n}{r u'^2}.
\end{align*}
Consequently
\begin{equation*}
 (M-\theta)'=-h_1(u)\frac{Q_n}{r u'^2}.
\end{equation*}
Moreover, $B_a=B_0-a(M-\theta)$.  Hence \eqref{eq:B0prime} yields
\begin{equation}\label{eq:Baprime}
 \begin{aligned}
 B_a'=-2F(u)\frac{Q_n}{r u'^2}
       +a h_1(u)\frac{Q_n}{r u'^2}
     =-2F_a(u)\frac{Q_n}{r u'^2}.
 \end{aligned}
\end{equation}
If $u(r)\geq1$, set
\[
 T_1(r):=Q(r)-g_1(u(r))M(r).
\]

The following identity replaces the model-specific connection formula
in \cite[Section~5]{Tang2026}.

\begin{lemma}
If $u'\ne0$, let $\theta$ be the quantity defined above.
Then, wherever $u\ne0$ and $u'\ne0$,
\begin{equation}\label{eq:connection}
 Q-P\frac vu=\omega(M-\theta).
\end{equation}
Moreover, with
\[
 T_2:=Q-g_2(u)M
\]
on an interval $\abs{u}>1$,
\begin{align}
 T_2'
 &=r^{n-1}g_2'(u)h_1(u)v-g_2'(u)u'M,\label{eq:T2prime-regular}\\
 &= -g_2'(u)u'(M-\theta)\quad \text{if }u'\neq0.\label{eq:T2prime}
\end{align}
Finally, wherever $u'\ne0$,
\begin{equation}\label{eq:T2B0}
 T_2=B_0+\frac{g_2(u)}{g_2'(u)u'}T_2'.
\end{equation}
\end{lemma}

\begin{proof}
Expanding the left-hand side gives
\begin{align*}
 Q-P\frac vu
 &=r^n\left(u'v'-\frac{u'^2v}{u}
       +f(u)v-2F(u)\frac vu\right)\\
 &=r^n u'\left(v'-\frac{u'v}{u}\right)
   +r^n\frac vu\bigl(uf(u)-2F(u)\bigr)\\
 &=\left(-\frac{ru'}u\right)
   \left[r^{n-1}(u'v-uv')
    -r^{n-1}\frac{h_1(u)v}{u'}\right]\\
 &=\omega(M-\theta).
\end{align*}
This proves \eqref{eq:connection}.  Since
$h_1'=sf'(s)-f(s)$ and $g_2=2F/h_1$,
\[
 g_2'(s)h_1(s)=2f(s)-g_2(s)\bigl(sf'(s)-f(s)\bigr).
\]
Using \eqref{eq:Mprime}--\eqref{eq:Qprime} therefore yields
\eqref{eq:T2prime-regular}.  Substitution of the definition of $\theta$
gives \eqref{eq:T2prime} away from critical points.  Moreover,
\[
 B_0=Q-g_2(u)\theta,\qquad
 T_2-B_0=-g_2(u)(M-\theta).
\]
The equality
\[
 M-\theta=-\frac{T_2'}{g_2'(u)u'}
\]
then gives \eqref{eq:T2B0} wherever $u'\ne0$.
\end{proof}

For the double power, the bridge factorization has an explicit form.

\begin{lemma}\label{lem:factor}
Fix $A\ge\beta$ and put $a=g_2(A)$.  Then $F_a(\pm A)=0$ and
\begin{equation*}
 -2F_a(s)=C_A\abs{s}^{q+1}
 \left[1-\left(\frac{\abs{s}}{A}\right)^{p-q}\right],
 \qquad \abs{s}\le A,
\end{equation*}
where
\[
 C_A:=\frac{2-(q-1)g_2(A)}{q+1}>0.
\]
\end{lemma}

\begin{proof}
By the definition, $F_a(\pm A)=F_a(A)=F(A)-g_2(A)h_1(A)/2=0$.  Since $F_a$ is a linear combination of the two powers $\abs{s}^{p+1}$ and $\abs{s}^{q+1}$, the vanishing at $A$ forces the stated factorization.  If $A=\beta$, then $g_2(A)=0$; if $A>\beta$, then $0<g_2(A)<2/(p-1)<2/(q-1)$.  Hence $C_A>0$ in either case.
\end{proof}

\section{First-phase variation}\label{sec:first}

\begin{theorem}\label{thm:firstphase}

Let $u=u(\cdot,\alpha)$ be a bound state of \eqref{eq:shoot}.
Assume that $u$ has a first zero $z_1$ and a
first postzero critical point $c_1>z_1$.
Let $r_1\in(0,z_1)$ be the point at which $u(r_1)=1$.  Then the variation
$v=\partial_\alpha u$ has exactly one zero $\tau_1\in(0,r_1)$
on $[0,c_1]$.  More precisely,
\[
 v>0\quad\text{on }[0,\tau_1),
 \qquad
 v<0\quad\text{on }(\tau_1,c_1].
\]
In addition,
\begin{equation*}
 Q(c_1),\ M(c_1)>0,
\end{equation*}
and
\begin{equation}\label{eq:first-Q12}
Q>0\quad\text{on }(0,r_1]\cup [\bar b_1,c_1],\qquad M,T_1>0\quad\text{on } (0,r_1],\qquad Q_1, Q_2>0
\quad \text{on } (0,c_1].
\end{equation}
Moreover
\begin{equation*}
  Q(\bar b_1)>Q(b_1)>0.
\end{equation*}
\end{theorem}

\begin{proof}
From Lemma~\ref{lem:nodal-geometry}, we obtain $\alpha>\beta$, $u(c_1)<-\beta$ and $u'<0$ on $(0,c_1)$.
This implies that
all four level points below are well defined and unique:
\begin{equation*}
 \begin{gathered}
 u(b_1)=\beta,\qquad u(r_1)=1,
 \qquad 0<b_1<r_1<z_1,\\
 u(\bar r_1)=-1,\qquad u(\bar b_1)=-\beta,
 \qquad z_1<\bar r_1<\bar b_1<c_1.
 \end{gathered}
\end{equation*}
Moreover, by $u(c_1)<-\beta$,
\begin{equation}\label{eq:firstphase-energy-positive}
 E(r)>E(c_1)=F(u(c_1))>0,\qquad 0\le r< c_1.
\end{equation}

We divide the argument into five steps.

\medskip\noindent
\textbf{Step 1: the first zero occurs before the level $u=1$.}
If $v$ had no zero on $(0,r_1)$, then $v>0$ on $[0,r_1)$ and, by continuity, $v(r_1)\ge0$.  Since $u>1$, \eqref{eq:high-signs} and \eqref{eq:rhoprime} give $\rho'>0$ on $(0,r_1)$, hence $\rho(r_1)>\rho(0)=0$.  But $f(1)=0$, $f'(1)=p-q>0$, and $u'(r_1)<0$, so
\[
 \rho(r_1)=r_1^{n-1}f'(1)u'(r_1)v(r_1)\le0,
\]
a contradiction.  Thus $v$ has a first zero $\tau_1\in(0,r_1)$, and $v'(\tau_1)<0$.

\medskip\noindent
\textbf{Step 2: there is no second zero while $u\geq1$.}
Before $\tau_1$, $f>0$, $uf'-f>0$, and $v>0$, so $Q',M'>0$.  Moreover
\[
 (r^{n-1}v')'=-r^{n-1}f'(u)v<0
\]
and $v'(0)=0$, hence $v'<0$.  Since $Q(0)=M(0)=T_1(0)=0$ and
\[
 T_1'=-g_1'(u)u'M,
\]
we obtain $Q,M,T_1>0$ on $(0,\tau_1]$.

Suppose that either $M$ ceases to be positive or $v$ has a second zero on
$(\tau_1,r_1]$, and let $s_*$ be the first such event.  At $\tau_1$, $M(\tau_1)>0$.
If $v$ had a second zero $\sigma\le r_1$ before $M$ vanished, then
$v'(\sigma)>0$ and
\[
 M(\sigma)=-\sigma^{n-1}u(\sigma)v'(\sigma)<0.
\]
Thus $M$ would already have vanished at an earlier point.  Hence the first
event must satisfy
\[
 M(s_*)=0,\qquad v(s_*)<0.
\]
On $(\tau_1,s_*)$, one has $M>0$, and therefore
$T_1'=-g_1'(u)u'M>0$.  It follows that
\[
 Q(s_*)=T_1(s_*)>0.
\]
The equality $M(s_*)=0$ gives $u'v=uv'$, and direct substitution yields
\[
 P_1(s_*)=Q(s_*)\frac{u(s_*)}{v(s_*)}.
\]
Here $u(s_*)\ge1$ and $v(s_*)<0$, while
$P_1(s_*)>0$ by Lemma~\ref{lem:Ppositive}.
The right-hand side is negative.  This contradiction excludes both events.
Consequently
\[
 M>0,\qquad v<0\quad\text{on }(\tau_1,r_1].
\]
Using $T_1'=-g_1'(u)u'M>0$ again, we obtain $T_1>0$ on $(\tau_1,r_1]$. Thus $Q=T_1+g_1M>0$ on $(\tau_1,r_1]$.

At $r_1$, $g_1(1)=0$, so
\begin{equation}\label{eq:Qr1-positive}
 Q(r_1)=T_1(r_1)>0.
\end{equation}
Since $u'v>0$ there, also $Q_1(r_1)>0$.

\medskip\noindent
\textbf{Step 3: the bridge $B_0$ rules out a second zero on $(r_1,z_1]$.}
Suppose that either $Q_1$ reaches zero or $v$ has a second zero on
$(r_1,z_1]$, and let $s_*$ be the first such event.  Before $s_*$,
\[
 v<0,\qquad Q_1>0,\qquad u'v>0,
\]
and hence $Q_n>Q_1>0$.  Since $F(u)<0$ on $(r_1,z_1)$,
\eqref{eq:B0prime} gives $B_0'>0$.  Also, by \eqref{eq:Qr1-positive},
\[
 B_0(r_1)=Q(r_1)-2F(1)\frac{r_1^{n-1}v(r_1)}{u'(r_1)}>0.
\]

The first event cannot be a zero of $v$.  Indeed, at a second zero, one has
$v'(s_*)>0$, and therefore
\[
 Q_1(s_*)=Q(s_*)=s_*^n u'(s_*)v'(s_*)<0.
\]
Thus $Q_1$ would have vanished strictly before $s_*$.
Consequently, the first event would have to satisfy
\[
 Q_1(s_*)=0,\qquad v(s_*)<0.
\]
Using $Q_1(s_*)=0$ in the definition of $B_0$, we obtain
\[
 B_0(s_*)=-2s_*^{n-1}\frac{v(s_*)}{u'(s_*)}E(s_*)<0,
\]
where \eqref{eq:firstphase-energy-positive} gives $E(s_*)>0$.  This contradicts
$B_0(s_*)>B_0(r_1)>0$.  Hence neither event occurs on $(r_1,z_1]$.
Consequently, $v$ has exactly one zero on $(0,z_1)$ and $v(z_1)<0$.

The same argument also records the $Q$-family signs.  On $(0,\tau_1]$, $Q_1',Q_2'>0$.  On $(\tau_1,r_1]$, $Q>0$ and $u'v>0$, hence $Q_2>Q_1>Q>0$.  On $(r_1,z_1]$, the bridge argument gives $Q_1>0$, and $Q_2>Q_1$ because $u'v>0$.  Thus
\begin{equation*}
 Q_1,Q_2>0\quad\text{on }(0,z_1].
\end{equation*}
In particular,
at $z_1$, the bridge identity $B_0(z_1)=Q(z_1)$ and the monotonicity from $r_1$ give
\begin{equation*}
Q(z_1)=B_0(z_1)>B_0(r_1)>0.
\end{equation*}

\medskip\noindent
\textbf{Step 4: passage through the postzero low-amplitude interval.}
On $(z_1,\bar b_1)$, one has $F<0$.  Suppose that either $Q_1$ reaches
zero or $v$ has a new zero on $(z_1,\bar b_1]$, and let $s_*$ be the first
event.  Before $s_*$, $v<0$, $Q_1>0$, and $u'v>0$,
so $Q_n>Q_1>0$.  Thus \eqref{eq:B0prime} gives $B_0'>0$ on $(z_1,\bar b_1)$.

If the first event were a zero of $v$, then $v'(s_*)>0$ and
\[
 Q_1(s_*)=Q(s_*)=s_*^n u'(s_*)v'(s_*)<0.
\]
Thus $Q_1$ would have reached zero earlier.  Hence the first event would have to satisfy $Q_1(s_*)=0$ with
$v(s_*)<0$.  As in Step~3,
\[
 B_0(s_*)=-2s_*^{n-1}\frac{v(s_*)}{u'(s_*)}E(s_*)<0,
\]
whereas $B_0(s_*)> B_0(z_1)=Q(z_1)>0$.  This contradiction proves
\[
 v<0,\qquad Q_2>Q_1>0\quad\text{on }(z_1,\bar b_1].
\]
In particular, since $F(-\beta)=0$,
\[
 Q(\bar b_1)=B_0(\bar b_1)>B_0(z_1)=Q(z_1)>0.
\]

We also record the comparison between the two points at the level $\abs{u}=\beta$.  If $\tau_1\leq b_1$, then $u'v>0$ and $Q_n> Q_1>0$ on $(b_1,\bar b_1)$.  Since $F<0$ there, \eqref{eq:B0prime} gives $B_0'>0$.  Hence
\[
 Q(\bar b_1)=B_0(\bar b_1)>B_0(b_1)=Q(b_1)>0.
\]
Suppose next that $b_1<\tau_1$.  On $[b_1,\tau_1)$, one has $f(u)>0$ and $v>0$, so \eqref{eq:Qprime} gives $Q'>0$.  At $\tau_1$, one has $B_0(\tau_1)=Q(\tau_1)$. While on $(\tau_1,\bar b_1)$, one has $u'v>0$ and $Q_n>Q_1>0$.
This with $F<0$
 gives $B_0'>0$ on $(\tau_1,\bar b_1)$.
Therefore
\[
 Q(\bar b_1)=B_0(\bar b_1)>B_0(\tau_1)=Q(\tau_1)>Q(b_1)>0.
\]
Thus $Q(\bar b_1)>Q(b_1)>0$ in both cases.

On $(\bar b_1,c_1]$, we have $u<-\beta$, hence $f(u)<0$.  A further zero $\sigma$ of $v$ would force $Q'=2r^{n-1}f(u)v>0$ on $(\bar b_1,\sigma)$ and $Q(\sigma)=\sigma^n u'(\sigma)v'(\sigma)<0$, which is impossible since $Q(\bar b_1)>0$.  Hence $v<0$ all the way to $c_1$, $Q_2(c_1)=Q_1(c_1)=Q(c_1)>0$, and $Q_2>Q_1>Q>0$ on $(\bar b_1,c_1)$, proving \eqref{eq:first-Q12}.

\medskip\noindent
\textbf{Step 5: recovery of $M(c_1)>0$.}
At $\bar r_1$, $f(-1)=0$ and $f'(-1)=p-q>0$.  Since $u'v>0$,
\[
 \rho(\bar r_1)=\bar r_1^{n-1}f'(-1)u'(\bar r_1)v(\bar r_1)>0.
\]
On $(\bar r_1,c_1)$, one has $u<-1$, hence $f''(u)<0$, while $v<0$; thus \eqref{eq:rhoprime} gives $\rho'>0$.  At $c_1$,
\[
 0<\rho(c_1)=-c_1^{n-1}f(u(c_1))v'(c_1).
\]
Since $f(u(c_1))<0$, we have $v'(c_1)>0$, and therefore
\[
 M(c_1)=-c_1^{n-1}u(c_1)v'(c_1)>0.
\]
This completes the proof.
\end{proof}

\section{Finite-phase transition}\label{sec:transition}
The phase-transition argument of \cite[Section~5]{Tang2026} requires control of \(Q_n\) on the part of each finite phase starting from the prezero \(F\)-level \(b\), characterized by
$$
F(u(b))=0,\qquad u(b)=\beta.
$$
For \(n\ge4\), and also for \(n=3\) with \(p+q\ge3\), the required pointwise positivity
$$
Q_n>0\quad\text{on }[b,c_+]
$$
follows from a direct algebraic contradiction; see Lemma~\ref{lem:Qn}. The only remaining parameter regime is
$$
n=3,\qquad p+q<3.
$$
In this residual case, \(Q_3\) need not be pointwise positive on the prezero interval \([b,\tau]\). Instead, we establish a weighted integral positivity estimate that is sufficient for the bridge argument and hence plays the role of pointwise positivity in the phase transition.

Let $u=u(\cdot,\alpha)$ be a bound state of \eqref{eq:shoot} and
$v=\partial_\alpha u$ be its variation.
Consider now an arbitrary finite phase \([c_-,c_+]\). Replacing the pair \((u,v)\) by \((-u,-v)\), if necessary, we may assume that the phase has the orientation specified below. This reflection entails no loss of generality. Indeed, \(f\) is odd, \(F\), \(h\), and \(h_1\) are even, \(h_2\) is odd, and the extension of \(g_2\) introduced above is even. Consequently, all phase functionals and identities used in the argument are invariant under the simultaneous transformation
$$
\widetilde u=-u,\qquad \widetilde v=-v.
$$
More precisely,
$$
 \widetilde P=P,\qquad
 \widetilde P_j=P_j,\qquad
 \widetilde\omega=\omega,\qquad
 \widetilde M=M,\qquad
 \widetilde Q=Q,\qquad
 \widetilde Q_j=Q_j,
$$
and
$$
 \widetilde\rho=\rho,\qquad
 \widetilde\theta=\theta,\qquad
 \widetilde T_2=T_2,\qquad
 \widetilde B_0=B_0.
$$
The same invariance holds for \(B_a\) when \(a=g_2(A)\). Hence every alternating finite phase may, without loss of information, be normalized so that
\begin{equation}\label{eq:phase-order}
 u(c_-)=A_->\beta,\qquad
 u'<0\quad\text{on }(c_-,c_+),\qquad
 u(c_+)=-A_+<-\beta,\qquad
 A_->A_+.
\end{equation}

Let
\[
0< c_-<b<r<z<\bar r<\bar b<c_+
\]
be determined by
\[
 u(b)=\beta,
 \qquad u(r)=1,
 \qquad u(z)=0,
 \qquad u(\bar r)=-1,
 \qquad u(\bar b)=-\beta.
\]
Suppose that at $c_-$,
\begin{equation}\label{eq:ind-data}
 Q(c_-)>0,
 \qquad M(c_-)>0,
 \qquad T_2(c_-)>0.
\end{equation}
We first justify carefully the analogue of Step 1 of Theorem~\ref{thm:firstphase}.  Since $u'(c_-)=0$ and $u(c_-)=A_->\beta>1$, the induction data \eqref{eq:ind-data} and the definitions \eqref{eq:M}--\eqref{eq:Q} give
\[
 Q(c_-)=c_-^n f(A_-)v(c_-)>0,
 \qquad
 M(c_-)=-c_-^{n-1}A_-v'(c_-)>0.
\]
Because $f(A_-)>0$ and $A_->0$, it follows that
\begin{equation}\label{eq:initial-v-signs}
 v(c_-)>0,
 \qquad
 v'(c_-)<0.
\end{equation}
In particular, using \eqref{eq:rho},
\begin{equation*}
 \rho(c_-)=-c_-^{n-1}f(A_-)v'(c_-)>0.
\end{equation*}

We claim that $v$ must vanish before the level $u=1$.  If not, then $v>0$ on $[c_-,r)$ and $v(r)\ge0$.  On $(c_-,r)$, one has $u>1$, and hence \eqref{eq:high-signs} and \eqref{eq:rhoprime} yield $\rho'(s)>0$.
Therefore, $\rho(r)>\rho(c_-)>0$.
On the other hand, $u(r)=1$, $f(1)=0$, $f'(1)=p-q>0$, and $u'(r)<0$, so by \eqref{eq:rho},
\[
 \rho(r)=r^{n-1}f'(1)u'(r)v(r)\le0,
\]
a contradiction.  Consequently, $v$ has a first zero
\begin{equation}\label{eq:tau-phase}
 \tau\in(c_-,r).
\end{equation}
Moreover, $v'(\tau)<0$.

It remains to record the signs before this first crossing.  For $c_-<s<\tau$, we have $u(s)>1$, $u'(s)<0$, and $v(s)>0$.  Moreover, by \eqref{eq:high-signs}, $f'(u(s))>0$, and the variational equation \eqref{eq:var} gives
\[
 \bigl(s^{n-1}v'(s)\bigr)'=-s^{n-1}f'(u(s))v(s)<0.
\]
Together with $v'(c_-)<0$ from \eqref{eq:initial-v-signs}, this implies
\[
 v'(s)<0\quad \text{on }(c_-,\tau).
\]
Since $u>1$ also gives $f(u)>0$, we conclude that throughout $(c_-,\tau)$,
\begin{equation*}
 uf(u),\
  f'(u), \
 uv, \
 u'v'>0,
 \qquad
 uu', \
  uv', \
  u'v, \
  vv'<0.
\end{equation*}

The following argument is the analogue of the first step in Tang's
phase-transition construction \cite[Lemma~5.1]{Tang2026}.

\begin{lemma}\label{lem:prep}
Assume \eqref{eq:ind-data}.  Then the first zero \eqref{eq:tau-phase} is the only zero of $v$ before $\max\{\tau,b\}$, and
\begin{equation}\label{eq:T2-prep}
 T_2,T_2'>0
 \quad\text{on }[c_-,\max\{\tau,b\}].
\end{equation}
Moreover
\begin{equation}\label{eq:Qfamily-prep}
 Q,M,  Q_1,  Q_2>0
 \quad\text{on }[c_-,\max\{\tau,b\}].
\end{equation}
\end{lemma}

\begin{proof}
On $[c_-,\tau]$, $Q$ and $M$
increase by \eqref{eq:Mprime}--\eqref{eq:Qprime}, thus $Q,M>0$.
At $c_-$, the identity
\eqref{eq:T2prime-regular} gives
\[
 T_2'(c_-)=c_-^{n-1}g_2'(u(c_-))h_1(u(c_-))v(c_-)>0.
\]
For $c_-<s\le\tau$, one has $u'<0$, $\theta\le0<M$, and
$g_2'(u)>0$.
Hence $T_2'>0$ on $[c_-,\tau]$, then $T_2>0$.

If $\tau\ge b$, $Q_1,Q_2$ increase by the sign record. Then $Q_j\geq Q_j(c_-)=Q(c_-)>0$ on $[c_-,\tau]$, $j=1,2$.  Suppose $\tau<b$ and assume that $v$ has a second zero $\widetilde\tau\le b$.  At the first zero,
\[
 0<T_2(\tau)=\tau^{n-1}u(\tau)v'(\tau)
 \bigl[g_2(u(\tau))-\omega(\tau)\bigr].
\]
Since $u(\tau)v'(\tau)<0$, we have $g_2(u(\tau))<\omega(\tau)$.  On $(\tau,\widetilde\tau)$, the amplitude decreases, so $g_2(u)$ decreases. On $\abs{u}>1$, Lemma~\ref{lem:Ppositive} gives $\omega'>0$.  At the second zero, $u(\widetilde\tau)v'(\widetilde\tau)>0$ and
 \begin{equation*}
 g_2(u(\widetilde\tau))-\omega(\widetilde\tau)<g_2(u(\tau))-\omega(\tau)<0,
\end{equation*}
hence the same formula would give $T_2(\widetilde\tau)<0$.
Since $T_2'(\tau)>0$, we know that,
before reaching $\widetilde\tau$, there is a first point $t_*$ at which $T_2'=0$ while $T_2>0$.  Formula \eqref{eq:T2prime} gives $M=\theta>0$ there.  Since $uv<0$ and $P>0$, the connection identity gives
\[
 Q=P\frac vu<0.
\]
But $T_2=Q-g_2M>0$ and $M>0$ imply $g_2<0$, equivalently $F<0$, whereas $t_*<b$ and hence $F>0$. Thus, $v$ has no zero on $(\tau,b]$. Moreover,
 this contradiction also shows that $T_2'$ cannot vanish before $b$.  Thus \eqref{eq:T2-prep} holds.

 On $(\tau,b]$, one has $u'v>0$ and $Q'=2r^{n-1}fv<0$, while $Q(b)=T_2(b)>0$ because $g_2(\beta)=0$.  Hence $Q>0$ on $(\tau,b]$, and therefore $Q_2>Q_1>Q>0$.

We also record that $M$ cannot vanish on $(\tau,b]$.  Indeed, at a first such zero $r_m$, the already established positivity of $T_2$ gives $Q(r_m)=T_2(r_m)>0$.  Since $M(r_m)=0$ means $u'v=uv'$, direct substitution in the definitions gives
\[
 P_1(r_m)=Q(r_m)\frac{u(r_m)}{v(r_m)}.
\]
Here $u>1$ and $v<0$, whereas $P_1>0$ by Lemma~\ref{lem:Ppositive}; this is impossible.  Hence $M>0$ on $(\tau,b]$.
\end{proof}

\begin{lemma}\label{lem:Qn}
Assume \eqref{eq:ind-data}. Suppose that $Q_1>0$ on $[c_-,c_+]$.  If either
\[
 n\ge4,
 \qquad\text{or}\qquad
 n=3\ \text{ and }\ p+q\ge3,
\]
then
\begin{equation}\label{eq:Qnpositive}
 Q_n>0\quad\text{on }[b,c_+].
\end{equation}
\end{lemma}

\begin{proof}
The assumption $Q_1>0$ excludes a second zero of $v$.  Indeed, at such a
zero, one would have $v'>0$ and
\[
 Q_1=r^n u'v'<0.
\]
Consequently, $v<0$ after $\tau$, and hence $u'v>0$ on $(\tau,c_+)$.  Thus
$Q_n>Q_1>0$ there.  At $\tau$ and $c_+$, one has
$Q_n=Q_1>0$.  Suppose that \eqref{eq:Qnpositive} fails.  By continuity,
$Q_n$ has a zero in $[b,\tau)$.  Let $s_0$ be its largest zero there.  Then
\[
 b\le s_0<\tau,\qquad Q_n(s_0)=0,
 \qquad Q_n>0\quad\hbox{on }(s_0,c_+].
\]

At $c_-$,
$\rho(c_-)>0$.  On $(c_-,\tau)$,
$\rho'>0$, so
$\rho(s_0)>0$.  At $s_0$, the equality $Q_n(s_0)=0$ reads
\[
 -u'v=\frac{s_0}{2(n-1)}\bigl(u'v'+f(u)v\bigr).
\]
The right-hand side is positive.  Moreover, $\rho(s_0)>0$ gives
\[
 -f(u)v'>-f'(u)u'v
 =\frac{s_0f'(u)}{2(n-1)}\bigl(u'v'+f(u)v\bigr)
 >\frac{s_0f'(u)}{2(n-1)}u'v'.
\]
Here all functions are evaluated at $s_0$.  Dividing by $-v'(s_0)>0$ and
using $\omega=-ru'/u$ yields
\begin{equation}\label{eq:add}
 \frac{2(n-1)f(u(s_0))}{u(s_0)f'(u(s_0))}>\omega(s_0).
\end{equation}
At $b$, $F(u(b))=0$ and $P(b)>0$, we obtain $\omega(b)>n-2$ from
  \begin{equation*}
    0<P(b)=b^nu'^2+(n-2)b^{n-1}uu'=b^{n-1}|u u'|\cdot[\omega(b)-(n-2)].
  \end{equation*}
  Since $b\leq s_0<r$, $1<u(s_0)\le\beta$ and $P_1>0$, so $\omega(s_0)\ge\omega(b)>n-2$.  With
\[
 x:=u(s_0)^{p-q}\in\left(1,\frac{p+1}{q+1}\right],
\]
from \eqref{eq:varpi} and \eqref{eq:add}, we obtain
\begin{equation}\label{eq:xineq}
 \bigl[2(n-1)-(n-2)p\bigr]x
 >2(n-1)-(n-2)q.
\end{equation}

For $n\ge4$, the coefficient of $x$ is positive under $p<(n+2)/(n-2)$.  Using $x\le(p+1)/(q+1)$, \eqref{eq:xineq} would imply
\[
 \bigl[2(n-1)-(n-2)p\bigr](p+1)
 >\bigl[2(n-1)-(n-2)q\bigr](q+1).
\]
The difference of the left side and the right side equals
\begin{equation}\label{eq:param-Qn}
 -(p-q)\bigl((n-2)(p+q)-n\bigr).
\end{equation}
For $n\ge4$, one has $p+q>2\ge n/(n-2)$, so \eqref{eq:param-Qn} is nonpositive and we obtain a contradiction.  This calculation explains the auxiliary threshold $p+q\ge n/(n-2)$, but also shows why it imposes no extra condition once $n\ge4$.

Let $n=3$.  If $p\ge4$, then \eqref{eq:xineq} reads
\[
 (4-p)x>4-q.
\]
For $q<4$, the right-hand side is positive whereas the left is nonpositive.  For $q\ge4$, since $x>1$ and $p>q$, one has $(4-p)x<4-p<4-q$.  Both cases are impossible.  If $p<4$, the coefficient $4-p$ is positive and the same endpoint estimate as above gives a contradiction precisely when $p+q\ge3$.  This proves the lemma.
\end{proof}

\begin{corollary}
\label{cor:Qn-first}
Let the notation be as in Theorem~\ref{thm:firstphase}.  If either
\[
 n\ge4,
 \qquad\text{or}\qquad
 n=3\ \text{ and }\ p+q\ge3,
\]
then
\[
 Q_n>0\quad\text{on }[b_1,c_1].
\]
\end{corollary}

\begin{proof}
Theorem~\ref{thm:firstphase} gives $Q_1>0$ on $(0,c_1]$ and
shows that $v$ has exactly one zero $\tau_1$.  If $\tau_1\le b_1$, then
$u'v\ge0$ on $[b_1,c_1]$, with equality only at $\tau_1$ and $c_1$.
Hence $Q_n\ge Q_1>0$ there.

Suppose that $b_1<\tau_1$.
Then $Q_n\geq Q_1>0$ on $[\tau_1,c_1]$.
 If the conclusion failed, continuity
would give a largest zero
$s_0\in[b_1,\tau_1)$.  On $(0,\tau_1)$, one has $u>1$, $v>0$, and
\[
 \rho(0)=0,\qquad
 \rho'=r^{n-1}f''(u)u'^2v>0.
\]
Thus $\rho(s_0)>0$.  Since $s_0\geq b_1$, the equality $Q_n(s_0)=0$ and the
preceding calculation give
\[
 \frac{2(n-1)f(u(s_0))}{u(s_0)f'(u(s_0))}>\omega(s_0)
 \ge\omega(b_1)>n-2.
\]
Hence \eqref{eq:xineq} holds.  The parameter calculation in
\eqref{eq:param-Qn} gives the same contradiction for $n\ge4$, or for
$n=3$ with $p+q\ge3$.
\end{proof}

\subsection{The three-dimensional residual range}

We now assume
\begin{equation}\label{eq:residual-range}
 n=3,\qquad 1<q<p,\qquad p+q<3.
\end{equation}
Then automatically
\begin{equation}\label{eq:residual-auto}
 1<q<\frac32,\qquad q<p<2.
\end{equation}
These are consequences of $p+q<3$ and $q<p$, not additional assumptions.  In a normalized phase, suppose $b<\tau$, so that a possible negative part of $Q_3$ lies in $[b,\tau)$. In fact, if $\tau\le b$, then
\[
 Q_3\ge Q_1>0\quad\text{on }[b,c_+],
\]
with equality possible only at $b=\tau$ or at the critical endpoint.
Thus no residual argument is needed.

\begin{lemma}\label{lem:residual-parameter}
Assume $1<q<p$ and $p+q<3$.
Then
\begin{equation}\label{eq:beta-basic-bounds}
 \beta<\sqrt e,
 \qquad
 {\beta^{q-1}(3\beta-4)<3\sqrt e-4<1.}
\end{equation}
Moreover, if
\begin{equation}\label{eq:kappa-window}
 1<\kappa<3(\beta-1),
\end{equation}
one has
\begin{equation}\label{eq:kappa-weight-bound}
 {\kappa \beta^{q-1}\bigl(3(\beta-1)-\kappa\bigr)
 <3\sqrt e-4<1.}
\end{equation}
\end{lemma}

\begin{proof}
The integral mean formula gives
\[
 \log \beta
 =\frac1{p-q}\int_q^p\frac{\dd t}{t+1}
 <\frac1{q+1}<\frac12,
\]
so $\beta<\sqrt e$.  By \eqref{eq:residual-auto}, $q<3/2$.  Put $t=(q+1)^{-1}\in(2/5,1/2)$.  If $3\beta\le4$, the second inequality in \eqref{eq:beta-basic-bounds} is immediate.  If $3\beta>4$, monotonicity in $\beta$ and the preceding exponential bound give
\begin{align*}
 \beta^{q-1}(3\beta-4)
 <e^{\frac{q-1}{q+1}}\left(3e^{\frac{1}{q+1}}-4\right)
 =3e^{1-t}-4e^{1-2t}
 <3\sqrt e-4<1.
\end{align*}
Here $3e^{1-t}-4e^{1-2t}$ is strictly increasing on $(2/5,1/2)$, because its derivative is $e^{1-2t}(8-3e^t)>0$ there.

Finally, $3(\beta-1)<3(\sqrt e-1)<2$.  Hence
\[
 y\longmapsto y\bigl(3(\beta-1)-y\bigr)
\]
is strictly decreasing for $y\ge1$.  Under \eqref{eq:kappa-window}, necessarily $\beta>4/3$, and therefore
\[
 \kappa\bigl(3(\beta-1)-\kappa\bigr)
 <3\beta-4.
\]
Multiplying by $\beta^{q-1}$ and using \eqref{eq:beta-basic-bounds} proves \eqref{eq:kappa-weight-bound}.
\end{proof}

\begin{lemma}\label{lem:residual}
Assume \eqref{eq:ind-data} and \eqref{eq:residual-range}. Consider a normalized phase with $b<\tau$.
Then for every $A\ge\beta$,
\begin{equation}\label{eq:residual-integral}
 \int_b^\tau
 u(r)^{q+1}
 \left[1-\left(\frac{u(r)}{A}\right)^{p-q}\right]
 \frac{Q_3(r)}{r u'(r)^2}\dd r>0.
\end{equation}
\end{lemma}

\begin{proof}
On $(b,\tau)$,
\[
 1<u< \beta,\qquad v>0,\qquad u'<0,\qquad v'<0,\qquad M>0.
\]
We organize the proof in seven steps.  They are the double-power
counterparts of the residual comparison 
\cite[Lemma~5.4]{Tang2026}.

\medskip\noindent
\textbf{Step 1: logarithmic slopes and the decreasing ratio $\mathcal R$.}
Since $P(b)>0$ and $F(\beta)=0$, we have $\omega(b)>n-2=1$.
Moreover $P_1>0$ on $u>1$, so $\omega$ is strictly increasing.  From $M>0$, we obtain
\begin{equation}\label{eq:aomega}
 a(r):=-\frac{rv'(r)}{v(r)}>\omega(r)>\omega(b)>1.
\end{equation}
Set
\[
 \mathcal R(r):=\frac{f(u(r))v(r)}{u'(r)v'(r)}>0.
\]
For $x=u^d$,
\[
 \frac{3f(u)}{u f'(u)}=\frac{3(x-1)}{px-q}.
\]
The right-hand side increases with $x$, and $x< \beta^d=(p+1)/(q+1)$; hence
\begin{equation*}
 \frac{3f(u)}{u f'(u)}<\frac{3}{p+q+1}<1<\omega.
\end{equation*}
A direct differentiation using the equations for $u$ and $v$ gives
\begin{align*}
 u'v'\mathcal R'
 &=f'u'v+fv'
 +\mathcal R\left(\frac4r u'v'+fv'+f'u'v\right)\\
 &<f'u'v+fv'+\frac4r fv\\
 &=\frac{fv}{r}\left(\frac{rv'}v+1\right)
 +\frac{f'uv}{r}\left(\frac{3f}{uf'}-\omega\right)<0.
\end{align*}
Since $u'v'>0$,
\begin{equation}\label{eq:Rdecrease}
 \mathcal R'(r)<0 \quad \text{on }(b,\tau).
\end{equation}

\medskip\noindent
\textbf{Step 2: the sign pattern of $Q_3$.}
For $n=3$,
\begin{equation}\label{eq:Q3prime}
 Q_3'=r^2\bigl(3u'v'-fv\bigr)
 =r^2u'v'(3-\mathcal R).
\end{equation}
If $Q_3\ge0$ on $[b,\tau]$, by \eqref{eq:Qfamily-prep},
\[
 Q_3(\tau)=Q(\tau)>0.
\]
Moreover
$u(\tau)<\beta\le A$, so all other factors in the integrand of
\eqref{eq:residual-integral} are strictly positive in a left neighborhood
of $\tau$.  Hence the integral is strictly positive.  Suppose therefore
that $Q_3$ is negative somewhere on $(b,\tau)$.  Equation \eqref{eq:Rdecrease} shows that
$3-\mathcal R$ is strictly increasing.  Thus $Q_3'$ has at most one zero.
If it has one, its sign changes from negative to positive.  The zero is
therefore the absolute minimum point $r_0$ of $Q_3$.  At this point
$\mathcal R(r_0)=3$ (equivalently, $3u'(r_0)v'(r_0)=f(u(r_0))v(r_0)$), and
\begin{align*}
 \frac{Q_3(r_0)}{r_0^2}=&r_0\bigl(u'(r_0)v'(r_0)+f(u(r_0))v(r_0)\bigr)+4u'(r_0)v(r_0)
 \\=&4r_0u'(r_0)v'(r_0)+4u'(r_0)v(r_0)
 =4u'(r_0)v(r_0)
 \left(1+\frac{r_0v'(r_0)}{v(r_0)}\right)>0
\end{align*}
by \eqref{eq:aomega}.  This value is positive.  It contradicts the
assumed negative value of $Q_3$, because $r_0$ would be its absolute
minimum.  Hence $Q_3'$ has no zero.  Since
$Q_3'(\tau)=3\tau^2u'(\tau)v'(\tau)>0$, it follows that $Q_3'>0$
on $(b,\tau)$.  The negative set is an initial interval.  The
positivity of $Q_3(\tau)=Q(\tau)$ now gives a unique
\begin{equation}\label{eq:sigma}
 \sigma\in(b,\tau),\qquad
 Q_3<0 \quad \text{on }[b,\sigma),\qquad
 Q_3(\sigma)=0,\qquad
 Q_3>0 \quad \text{on }(\sigma,\tau].
\end{equation}
Since $u'v'>0$ and $Q_3'>0$ on $(b,\tau)$, identity
\eqref{eq:Q3prime} also gives
\begin{equation}\label{eq:R-window}
 0<\mathcal R(r)<3\quad \text{on }(b,\tau).
\end{equation}

\medskip\noindent
\textbf{Step 3: the speed $\abs{u'}$ decreases on the residual interval.}
Lemma~\ref{lem:inflection} gives a unique inflection point on the prezero
interval $u>1$.  Before this point $u''<0$, and after it $u''>0$.  If the
point were at or to the right of $b$, then $u''(b)\le0$.  The equation for
$u$ would give
\begin{equation*}
 -\frac{2u'(b)v(b)}{b}\le f(\beta)v(b).
\end{equation*}
Since $Q_3(b)<0$ by \eqref{eq:sigma},
\[
 b\bigl(u'v'+fv\bigr)<-4u'v\le2b fv,
\]
and hence $u'v'+fv>2u'v'>0$.  On the other hand,
$\rho(c_-)>0$, and on $(c_-,b)$,
$\rho'>0$, so
$\rho(b)>0$. Then
\[
 -fv'>-f'u'v
 >\frac{b f'}4\bigl(u'v'+fv\bigr)
 >\frac{b f'}2u'v'.
\]
Dividing by $-v'(b)>0$
yields
\[
 \frac{2f(\beta)}{\beta f'(\beta)}>\omega(b)>1.
\]
But $\beta^d=(p+1)/(q+1)$ gives the exact identity
\begin{equation*}
 \frac{2f(\beta)}{\beta f'(\beta)}=\frac{2}{p+q+1}<1,
\end{equation*}
a contradiction.  Thus the inflection point lies strictly before $b$, and
\begin{equation}\label{eq:speed-decrease}
 \abs{u'(r)}\text{ is strictly decreasing on }(b,\tau).
\end{equation}

\medskip\noindent
\textbf{Step 4: quantitative separation of the positive and negative $Q_3$ regions.}
Since $v'<0$ and
\[
 (r^2v')'=-r^2f'(u)v<0,
\]
the function $r^2\abs{v'}$ is strictly increasing.  Integrating \eqref{eq:Q3prime} from $\sigma$ to $\tau$ gives
\begin{align*}
\tau^3\abs{u'(\tau)v'(\tau)}= Q_3(\tau)
 <3\int_\sigma^\tau r^2\abs{u'v'}\dd r<3\tau^2\abs{v'(\tau)}\bigl(u(\sigma)-u(\tau)\bigr).
\end{align*}
Therefore
\begin{equation}\label{eq:S6general}
 3\bigl(u(\sigma)-u(\tau)\bigr)>\tau\abs{u'(\tau)}.
\end{equation}
Put
\[
 S:=u(\sigma)> T:=u(\tau)>1,\qquad \omega(\tau)>\omega(b)>1,
 \qquad \kappa:=\frac{\omega(\tau)}{\omega(b)}>1.
\]
Then \eqref{eq:S6general} gives
\begin{equation*}
S-T>\frac{\tau\abs{u'(\tau)}}{3}>\frac{\tau\abs{u'(\tau)}}{3u(\tau)\omega(b)}=-\frac{\tau{u'(\tau)}}{3u(\tau)\omega(b)}=\frac{\kappa}{3}.
\end{equation*}
 In particular, $\kappa<3(\beta-1)$ because
$S-T<\beta-T<\beta-1$.  Also
\[
 (\beta-S)+(S-T)=\beta-T<\beta-1.
\]
Thus $\beta-S<\beta-1-\kappa/3$, and hence
\begin{equation*}
 \frac{S-T}{\beta-S}
 >\frac{\kappa}{3(\beta-1)-\kappa},
 \qquad 1<\kappa<3(\beta-1).
\end{equation*}
The denominator is positive.  Lemma~\ref{lem:residual-parameter} gives
$\kappa\beta^{q-1}[3(\beta-1)-\kappa]<1$. Hence
\begin{equation}\label{eq:weighted-length}
 \frac{S-T}{\beta-S}
 >\kappa^2\beta^{q-1}.
\end{equation}
This is the point at which the double-power proof differs
quantitatively from the pure-power residual estimate in
\cite[Lemma~5.4]{Tang2026}: the factor $\beta^{q-1}$ is exactly what is
needed to absorb the additional weight below.

\medskip\noindent
\textbf{Step 5: equal-amplitude pairing of $Q_3$.}
The bound \eqref{eq:R-window} gives $0<\mathcal R(\sigma)<3$.
For $r^-\in[b,\sigma)$, define $r^+\in(\sigma,\tau)$ by
\[
 u(r^-)-S=S-u(r^+).
\]
This point exists and is unique.  Indeed,
\eqref{eq:weighted-length} implies $S-T>\beta-S$. Since $\mathcal{R}$ decreases and $r^2|v'|$ increases,
by $\dd u=u'\dd r$, integration of \eqref{eq:Q3prime} gives
\begin{align*}
 \abs{Q_3(r^-)}=\int^\sigma_{r^-} r^2u'v'(3-\mathcal R)\dd r<\bigl(3-\mathcal R(\sigma)\bigr)\int^\sigma_{r^-} r^2u'v'\dd r
 &<\bigl(3-\mathcal R(\sigma)\bigr)\sigma^2\abs{v'(\sigma)}
 \abs{u(r^-)-S},\\
 Q_3(r^+)=\int_\sigma^{r^+} r^2u'v'(3-\mathcal R)\dd r>\bigl(3-\mathcal R(\sigma)\bigr)\int_\sigma^{r^+} r^2u'v'\dd r
 &>\bigl(3-\mathcal R(\sigma)\bigr)\sigma^2\abs{v'(\sigma)}
 \abs{S-u(r^+)},
\end{align*}
and hence
\begin{equation*}
 Q_3(r^+)>\abs{Q_3(r^-)}.
\end{equation*}

Define normalized points $\widetilde r_x\in[b,\sigma]$ and $\widehat r_x\in[\sigma,\tau]$ by
\[
 u(\widetilde r_x)=S+(\beta-S)x,
 \qquad
 u(\widehat r_x)=S-(S-T)x,
 \qquad 0\le x\le1.
\]
For $x>0$, the amplitude drop from $S$ to $u(\widehat r_x)$ is
$(S-T)x$.  The drop paired with $\widetilde r_x$ is $(\beta-S)x$.  Since
$S-T>\beta-S$, $\widehat r_x$ lies strictly to the right of that
equal-amplitude partner.  The strict increase of $Q_3$ and the preceding
paired estimate give
\begin{equation}\label{eq:normalized-pair}
 Q_3(\widehat r_x)>\abs{Q_3(\widetilde r_x)},\qquad
 0<x\le1.
\end{equation}
\textbf{Step 6: the weighted $Q_3$ mass and its quantitative margin.}
Use the normalized points from Step~5 and set
\[
 U_-(x)=S+(\beta-S)x,\qquad U_+(x)=S-(S-T)x.
\]
The changes of variable $\dd r=\dd u/u'$ give
\begin{align*}
 \int_b^\sigma u^{q-1}\abs{Q_3}\dd r
 &=(\beta-S)\int_0^1
 \frac{U_-(x)^{q-1}\abs{Q_3(\widetilde r_x)}}
 {\abs{u'(\widetilde r_x)}}\dd x,\\
 \int_\sigma^\tau u^{q-1}Q_3\dd r
 &=(S-T)\int_0^1
 \frac{U_+(x)^{q-1}Q_3(\widehat r_x)}
 {\abs{u'(\widehat r_x)}}\dd x.
\end{align*}
For $0<x\le1$, the speed decrease in \eqref{eq:speed-decrease} gives
\[
 \abs{u'(\widetilde r_x)}>
 \abs{u'(\sigma)}>
 \abs{u'(\widehat r_x)}.
\]
Also $U_-(x)\le\beta$ and $U_+(x)\ge T>1$, so
\[
 U_+(x)^{q-1}>\beta^{-(q-1)}U_-(x)^{q-1}.
\]
Together with \eqref{eq:normalized-pair}, these inequalities yield
\begin{align*}
 \int_\sigma^\tau u^{q-1}Q_3\dd r
 >\frac{S-T}{\beta^{q-1}\abs{u'(\sigma)}}
   \int_0^1 U_-(x)^{q-1}\abs{Q_3(\widetilde r_x)}\dd x
 >\frac{S-T}{\beta^{q-1}(\beta-S)}
   \int_b^\sigma u^{q-1}\abs{Q_3}\dd r.
\end{align*}
The integrands are strictly positive for $x>0$, so both integrated
inequalities are strict. By \eqref{eq:weighted-length},
\begin{equation}\label{eq:weighted-Q3-mass}
 \int_\sigma^\tau u^{q-1}Q_3\dd r
 >\kappa^2\int_b^\sigma u^{q-1}\abs{Q_3}\dd r.
\end{equation}

\medskip\noindent
\textbf{Step 7: insertion of the bridge weight.}
Put
\[
 K_A(u):=1-\left(\frac{u}{A}\right)^d.
\]
Since $A\ge \beta>S$, $K_A(S)>0$, and $K_A(u(r))$ increases with $r$ on
$(b,\tau)$.  Using $u'^2=\omega^2u^2/r^2$, the integral in
\eqref{eq:residual-integral} is
\[
 \int_b^\tau u^{q-1}K_A(u)\frac{rQ_3}{\omega^2}\dd r.
\]
On $[b,\sigma]$, one has $Q_3\le0$ and
\[
 \frac{rK_A(u)}{\omega^2}
 \le\frac{\sigma K_A(S)}{\omega(b)^2}
 =\frac{\sigma K_A(S)\kappa^2}{\omega(\tau)^2}.
\]
On $[\sigma,\tau]$, one has $Q_3\ge0$ and
\[
 \frac{rK_A(u)}{\omega^2}
 \ge\frac{\sigma K_A(S)}{\omega(\tau)^2}.
\]
The inequalities are strict away from $\sigma$.  Splitting the integral at
$\sigma$ and respecting the opposite signs of $Q_3$ therefore gives
\begin{align*}
 \int_b^\tau u^{q-1}K_A(u)\frac{rQ_3}{\omega^2}\dd r>\frac{\sigma K_A(S)}{\omega(\tau)^2}
 \left[
 \int_\sigma^\tau u^{q-1}Q_3\dd r
 -\kappa^2\int_b^\sigma u^{q-1}\abs{Q_3}\dd r
 \right]>0
\end{align*}
by \eqref{eq:weighted-Q3-mass}.  This proves \eqref{eq:residual-integral}.
\end{proof}

\begin{corollary}
\label{addcor1}
Assume \eqref{eq:residual-range} and the hypotheses of
Theorem~\ref{thm:firstphase}. \emph{If $b_1<\tau_1$}, then for every
$A\ge\beta$,
\[
 \int_{b_1}^{\tau_1}
 u^{q+1}\left[1-\left(\frac{u}{A}\right)^{p-q}\right]
 \frac{Q_3}{r u'^2}\dd r>0.
\]
If instead $\tau_1\le b_1$, one has the simpler pointwise estimate
\[
 Q_3\ge Q_1>0\quad\text{on }[b_1,c_1].
\]
\end{corollary}
\begin{proof}
If $\tau_1\le b_1$, Theorem~\ref{thm:firstphase} gives
$v\le0$ on $[b_1,c_1]$ and $u'<0$ on $[b_1,c_1)$; hence
$Q_3=Q_1+2r^2u'v\ge Q_1>0$.
Now assume $b_1<\tau_1$.  Theorem~\ref{thm:firstphase} gives
$Q,M>0$ on $[b_1,\tau_1]$, while
\[
 \rho(0)=0,\qquad
 \rho'=r^{n-1}f''(u)u'^2v>0
 \quad\text{on }(0,b_1].
\]
Thus $\rho(b_1)>0$.  These are the initial signs used in
Lemma~\ref{lem:residual}; its proof applies on the interval $[b_1,\tau_1]$ without any other changes.
\end{proof}

\subsection{Phase renewal}

We now close the induction.  All ingredients except the bridge
factorization are adaptations of the comparison mechanism in
\cite[Section~5]{Tang2026}.  By the reflection invariance recorded
before \eqref{eq:phase-order}, it is enough to work with the normalized
orientation \eqref{eq:phase-order}; the conclusion then transfers
unchanged to every alternating phase.

\begin{lemma}\label{lem:river}
Under \eqref{eq:ind-data}, $v$ has no second zero on $[\tau,c_+]$, and
\begin{equation}\label{eq:Q12all}
 Q_1,Q_2>0\quad\text{on }[c_-,c_+].
\end{equation}
Moreover
\begin{equation}\label{eq:Qbb}
 Q(\bar b)>Q(b)>0,
 \qquad
 Q>0\quad\text{on }[c_-,\max\{\tau,b\}]\cup[\bar b,c_+].
\end{equation}
\end{lemma}
\begin{proof}
Put $s_0:=\max\{\tau,b\}\geq b$.  Lemma~\ref{lem:prep} gives $Q,Q_1,Q_2>0$ on $[c_-,s_0]$ and
excludes a second zero of $v$ before $s_0$.
The simultaneous continuation argument in Steps~3--4 of
Theorem~\ref{thm:firstphase} applies on $(s_0,\bar b]$.
Indeed, its initial and interval hypotheses are
\[
 B_0(s_0)=Q(s_0)>0,\qquad
 F(u)<0,\qquad E(r)>E(c_+)>0.
\]
As long as $v<0$ and $Q_1>0$, we have $Q_n>Q_1>0$ and
$B_0'>0$. At a second zero of $v$, $Q_1=r^nu'v'<0$;
at a first zero of $Q_1$ with $v<0$,
$B_0=-2r^{n-1}(v/u')E<0$. Both alternatives contradict the
preceding signs. Therefore
\begin{equation}\label{add4.71}
  v<0,\qquad Q_2>Q_1>0
 \quad\text{on }(s_0,\bar b].
\end{equation}
Moreover,
\begin{equation}\label{add4.72}
  Q(\bar b)=B_0(\bar b)>B_0(s_0)>0.
\end{equation}

On $(\bar b,c_+]$, a further zero $\sigma$ of $v$ would force $Q'>0$ on $(\bar b,\sigma)$ and $Q(\sigma)<0<Q(\bar b)$, which is impossible.  Hence $v<0$ on $(\bar b,c_+]$, $Q_2(c_+)=Q_1(c_+)=Q(c_+)>0$, and $Q_2>Q_1>Q>0$ on $(\bar b,c_+)$.
This together with \eqref{eq:Qfamily-prep}, \eqref{add4.71} and \eqref{add4.72}, proves
\eqref{eq:Q12all} and the last assertion in \eqref{eq:Qbb}.

It remains to compare $Q(b)$ and $Q(\bar b)$.  If $\tau\le b$, then
$Q_n>Q_1>0$ on $(b,\bar b)$, so \eqref{eq:B0prime} and $F<0$ give $Q(\bar b)=B_0(\bar b)>B_0(b)=Q(b)>0$.
If $b<\tau$, then $f(u)>0$ and $v>0$ on $[b,\tau)$, hence $Q(\tau)>Q(b)>0$.
On $(\tau,\bar b)$, $Q_n>Q_1>0$, and the same bridge argument yields $Q(\bar b)=B_0(\bar b)>B_0(\tau)=Q(\tau)$.
Thus in both cases $Q(\bar b)>Q(b)>0$.
\end{proof}

\begin{lemma}\label{lem:phi}
Assume \eqref{eq:ind-data}.
For $\mu\in[\beta,A_+)$, let $r^-_\mu\in(c_-,b]$ and
$r^+_\mu\in[\bar b,c_+)$ be defined by
\[
 u(r^-_\mu)=\mu,
 \qquad
 u(r^+_\mu)=-\mu.
\]
Set
\[
 \Phi(r):=\frac{Q(r)}{r^{n-1}\abs{u'(r)}}.
\]
Then
\begin{equation}\label{eq:Phi-comp}
 \Phi(r^+_\mu)>\Phi(r^-_\mu)>0
 \quad\text{for all }\mu\in[\beta,A_+).
\end{equation}
\end{lemma}

\begin{proof}
Let
\[
 \widehat E(r):=r^{2(n-1)}E(r).
\]
Using \eqref{eq:Eprime},
\begin{equation*}
 \widehat E'(r)=2(n-1)r^{2n-3}F(u(r)).
\end{equation*}
Since $F<0$ between $b$ and $\bar b$, and $F(\pm\beta)=0$,
\begin{equation*}
 (r^+_\beta)^{n-1}\abs{u'(r^+_\beta)}
 <(r^-_\beta)^{n-1}\abs{u'(r^-_\beta)}.
\end{equation*}
Together with $Q(\bar b)>Q(b)>0$ from Lemma~\ref{lem:river}, this proves \eqref{eq:Phi-comp} at $\mu=\beta$.

For $\mu\in[\beta,A_+)$, define
\[
 D(\mu):=\Phi(r^+_\mu)-\Phi(r^-_\mu)
\]
and set
\[
 \Phi_2(r):=\frac{Q_2(r)}{r^{n-1}\abs{u'(r)}^3},
 \qquad
 D_2(\mu):=\Phi_2(r^+_\mu)-\Phi_2(r^-_\mu).
\]
We have already proved that $D(\beta)>0$.  We first show that
\begin{equation}\label{eq:Phi2-comp}
 D(\mu)>0\quad\Longrightarrow\quad D_2(\mu)>0.
\end{equation}
At equal amplitudes, strict energy decrease gives
\begin{equation}\label{eq:speed-pair}
 \abs{u'(r^+_\mu)}<\abs{u'(r^-_\mu)}.
\end{equation}
If $r^-_\mu\le\tau$, then $u'v\le0$ at $r^-_\mu$, while $u'v>0$ at
$r^+_\mu$.  Hence
\[
 Q_2(r^-_\mu)\le Q(r^-_\mu),
 \qquad
 Q_2(r^+_\mu)>Q(r^+_\mu),
\]
and
$D(\mu)>0$,
\eqref{eq:speed-pair} give the explicit chain
\begin{align*}
 \Phi_2(r^+_\mu)
 &>\frac{Q(r^+_\mu)}{(r^+_\mu)^{n-1}
 \abs{u'(r^+_\mu)}^3}
 =\frac{\Phi(r^+_\mu)}{\abs{u'(r^+_\mu)}^2}\\
 &>\frac{\Phi(r^-_\mu)}{\abs{u'(r^-_\mu)}^2}
 =\frac{Q(r^-_\mu)}{(r^-_\mu)^{n-1}
 \abs{u'(r^-_\mu)}^3}
 \ge\Phi_2(r^-_\mu).
\end{align*}
Thus $D_2(\mu)>0$ in this case.

Assume next that $r^-_\mu>\tau$.  Then $u'v>0$ at $r^-_\mu$ and $r^+_\mu$.
Since $Q_1>0$ on $[c_-,c_+]$ by Lemma~\ref{lem:river}, integration of \eqref{eq:vuratio} from
$r^-_\mu$ to $r^+_\mu$ gives
\begin{equation}\label{eq:vu-pair}
 \frac{v(r^+_\mu)}{u'(r^+_\mu)}
 >\frac{v(r^-_\mu)}{u'(r^-_\mu)}>0.
\end{equation}
At every point where $u'v>0$, one has
\[
 \Phi_2
 =\frac{\Phi}{\abs{u'}^2}
 +\frac{2}{\abs{u'}}\frac{v}{u'}.
\]
The first summand at $r^+_\mu$ is larger than the first summand at
$r^-_\mu$, because $D(\mu)>0$ and
\eqref{eq:speed-pair} hold.  The second summand is also larger, by
 \eqref{eq:speed-pair} and \eqref{eq:vu-pair}.  Hence
$D_2(\mu)>0$.  This proves
\eqref{eq:Phi2-comp}.

Suppose now that \eqref{eq:Phi-comp} fails.  By continuity, there is a first
contact $\mu_*\in(\beta,A_+)$ such that
\[
 D(\mu)>0\quad \text{on }(\beta,\mu_*),
 \qquad D(\mu_*)=0.
\]
On the prezero branch, parametrized by $s=u>0$, identity
\eqref{eq:Qratio} gives
\[
 \frac{\dd}{\dd s}\Phi(r^-_s)=f(s)\Phi_2(r^-_s).
\]
On the postzero branch, parametrized by $s=\abs{u}=-u>0$, the oddness of $f$ gives the same formula
\[
 \frac{\dd}{\dd s}\Phi(r^+_s)=f(s)\Phi_2(r^+_s).
\]
Consequently
\begin{equation}\label{eq:Phi-integral}
 D(\mu_*)=D(\beta)+\int_\beta^{\mu_*}f(s)D_2(s)\dd s.
\end{equation}
For $\beta<s<\mu_*$, one has $f(s)>0$, $D(s)>0$, and hence
$D_2(s)>0$ by \eqref{eq:Phi2-comp}.  The right-hand side of
\eqref{eq:Phi-integral} is therefore strictly positive.  This contradicts
$D(\mu_*)=0$, thus $\Phi(r^+_\mu)>\Phi(r^-_\mu)$ for all $\mu\in[\beta,A_+)$.
Finally, it follows from Lemma~\ref{lem:river} that $\Phi(r^-_\mu)>0$, and
proves \eqref{eq:Phi-comp}.
\end{proof}

\begin{lemma}\label{lem:renewal}
Assume \eqref{eq:ind-data}.
Then $T_2(c_+)>0$.
\end{lemma}

\begin{proof}
On the postzero high-amplitude interval $[\bar b,c_+)$,
$u'v>0$, so $Q_n>Q_1>0$.
At $c_+$, one has $Q_n=Q_1=Q>0$.
At $\bar b$, $g_2(-\beta)=0$, hence $T_2(\bar b)=Q(\bar b)>0$.  Since
\[
 u(c_+)<-\beta,\qquad v(c_+)<0,\qquad g_2'(u(c_+))<0,\qquad h_1(u(c_+))>0,
\]
by \eqref{eq:T2prime-regular},
\[
 T_2'(c_+)
 =c_+^{n-1}g_2'(u(c_+))h_1(u(c_+))v(c_+)>0.
\]
Thus $T_2'>0$ on a left neighborhood of $c_+$.

If $t\in(\bar b,c_+)$ satisfies $T_2'(t)=0$, then $u'(t)\ne0$.
Differentiating \eqref{eq:T2B0}, using \eqref{eq:B0prime} and
$T_2'(t)=0$, first gives
\[
 0=B_0'(t)+\frac{g_2(u(t))}{g_2'(u(t))u'(t)}T_2''(t).
\]
Since $2F=g_2h_1$, this becomes
\begin{equation*}
 T_2''(t)=h_1(u(t))g_2'(u(t))u'(t)
 \frac{Q_n(t)}{t u'(t)^2}>0.
\end{equation*}
Indeed, $h_1(u(t))>0$ and $Q_n(t)>0$, while both $g_2'(u(t))$ and
$u'(t)$ are negative on this branch.  Hence every zero of $T_2'$ is
simple, and $T_2'$ crosses there from negative to positive.  There can be
at most one such zero.  Otherwise, after the first zero $t_1$, the derivative
$T_2'$ is positive.  If $t_2>t_1$ were the next zero, then $T_2'>0$
immediately to the left of $t_2$, whereas $T_2''(t_2)>0$ would force
$T_2'<0$ immediately to its left.  This is a contradiction.  If $T_2'$
has no zero in $(\bar b,c_+)$, its sign is constant there.  Its positivity
near $c_+$ implies $T_2'>0$ throughout this interval.  Hence
$T_2(c_+)>T_2(\bar b)>0$.

Suppose therefore that the unique critical point exists and denote it by
$t^+$.  It lies strictly between $\bar b$ and $c_+$.  The crossing
orientation gives
\begin{equation}\label{eq:T2postsign}
 T_2'<0\quad\text{on }(\bar b,t^+),
 \qquad T_2'(t^+)=0,\qquad
 T_2'>0\quad\text{on }(t^+,c_+).
\end{equation}
Put
\[
 A:=\abs{u(t^+)}\in(\beta,A_+),
\]
and let $t^-\in(c_-,b)$ be the unique point with $u(t^-)=A$.  This point
exists because $A<A_+<A_-$ and $u$ is strictly decreasing.  Lemma~\ref{lem:prep}
gives $T_2(t^-)>0$.

Set $a=g_2(A)$ and
\begin{equation}\label{add4.92}
 H_A(s):=-2F_a(s)
 =C_A\abs{s}^{q+1}\left[1-\left(\frac{\abs{s}}{A}\right)^d\right]
 \begin{cases}
 =0,\qquad&\abs{s}=0, A,\\
 >0,\qquad&0<\abs{s}<A,
 \end{cases}
\end{equation}
where $C_A>0$ by Lemma~\ref{lem:factor}.
Since $F_a(\pm A)=0$, integration of \eqref{eq:Baprime} gives
\begin{equation}\label{eq:T2integral}
 T_2(t^+)-T_2(t^-)
 =\int_{t^-}^{t^+}H_A(u(r))\frac{Q_n(r)}{r u'(r)^2}\dd r.
\end{equation}
We show that the integral is strictly positive.

First pair the high-amplitude pieces $[t^-,b]$ and $[\bar b,t^+]$.  For $\mu\in[\beta,A]$, let $r^-_\mu$ and $r^+_\mu$ be the equal-amplitude points from Lemma~\ref{lem:phi} and set
\[
 \Psi(\mu):=
 \frac{Q_n(r^+_\mu)}{r^+_\mu\abs{u'(r^+_\mu)}^3}
 +\frac{Q_n(r^-_\mu)}{r^-_\mu\abs{u'(r^-_\mu)}^3}.
\]
We claim that
\begin{equation}\label{eq:Psi-positive}
 \Psi(\mu)>0\quad\text{for all }\mu\in(\beta,A).
\end{equation}
If $Q_n(r^-_\mu)\ge0$, this is immediate because $Q_n(r^+_\mu)>0$.  Assume $Q_n(r^-_\mu)<0$.  Then $r^-_\mu<\tau$.  On the prezero side, one has $\theta<0<M$, so the connection identity \eqref{eq:connection} and $uv>0$, $\omega>0$, $P,Q>0$ imply
\begin{equation}\label{eq:connection-pre-pair}
{v(r^-_\mu)}
 <\frac{Q(r^-_\mu)}{P(r^-_\mu)}\,\mu.
\end{equation}
On the postzero interval $(\bar b,t^+)$, \eqref{eq:T2postsign} gives $T_2'<0$.
Since $g_2'(u)u'>0$, formula \eqref{eq:T2prime} gives $M>\theta$.  There $\omega<0$; hence \eqref{eq:connection} gives
\begin{equation}\label{eq:connection-pre-pair'}
 \abs{v(r^+_\mu)}
 >\frac{Q(r^+_\mu)}{P(r^+_\mu)}\,\mu,
 \qquad \beta<\mu< A.
\end{equation}
Moreover, $Q_2(r^-_\mu)>0$ gives
\begin{equation}\label{add4.91}
 -Q_n(r^-_\mu)
 <(n-2)(r^-_\mu)^{n-1}\abs{u'(r^-_\mu)v(r^-_\mu)},
\end{equation}
whereas $u'v>0$ and $Q>0$ on the postzero side give
\begin{equation}\label{eq:Qn-post-bound}
 Q_n(r^+_\mu)
 >n(r^+_\mu)^{n-1}{u'(r^+_\mu)v(r^+_\mu)}.
\end{equation}

It remains to compare the coefficients in these two estimates.  Define
\[
 \mathcal S(r):=\frac{r^{n-2}Q(r)}{P(r)u'(r)^2}
 =\Phi(r)\,
 \frac{r^{n-3}}{\abs{u'(r)}}\,
 \frac{r^n}{P(r)}.
\]
At equal amplitudes, Lemma~\ref{lem:phi} gives $\Phi(r^+_\mu)>\Phi(r^-_\mu)>0$; the energy decrease gives $\abs{u'(r^+_\mu)}<\abs{u'(r^-_\mu)}$; $P>0$ and $P/r^n$ is decreasing by Lemma~\ref{lem:Ppositive}.  Since $r^+_\mu>r^-_\mu$ and $n\ge3$,
\begin{equation}\label{eq:S-pair}
 \mathcal S(r^+_\mu)>\mathcal S(r^-_\mu)>0.
\end{equation}
Combining \eqref{eq:connection-pre-pair}--\eqref{eq:Qn-post-bound} with \eqref{eq:S-pair} yields
\begin{align*}
 \frac{Q_n(r^+_\mu)}{r^+_\mu\abs{u'(r^+_\mu)}^3}
 >n\mu\,\mathcal S(r^+_\mu)
 >(n-2)\mu\,\mathcal S(r^-_\mu)
 >\frac{-Q_n(r^-_\mu)}{r^-_\mu\abs{u'(r^-_\mu)}^3},
\end{align*}
which proves \eqref{eq:Psi-positive}.

On $[t^-,b]$, the variable $\mu=u$ decreases from $A$ to $\beta$ and
$\dd r=-\dd\mu/\abs{u'}$.  On $[\bar b,t^+]$, the variable
$\mu=-u$ increases from $\beta$ to $A$ and
$\dd r=\dd\mu/\abs{u'}$.
By \eqref{add4.92} and \eqref{eq:Psi-positive}, using the evenness of $H_A$, we obtain
\begin{align}\label{eq:high-pair-positive}
 &\int_{t^-}^{b}H_A(u)\frac{Q_n}{r u'^2}\dd r
 +\int_{\bar b}^{t^+}H_A(u)\frac{Q_n}{r u'^2}\dd r
 \nonumber\\
 =&\int_A^\beta
 H_A(\mu)\frac{Q_n(r^-_\mu)}{r^-_\mu\abs{u'(r^-_\mu)}^2}
 \left(-\frac{\dd\mu}{\abs{u'(r^-_\mu)}}\right)+\int^A_\beta
 H_A(\mu)\frac{Q_n(r^+_\mu)}{r^+_\mu\abs{u'(r^+_\mu)}^2}
 \left(\frac{\dd\mu}{\abs{u'(r^+_\mu)}}\right)\nonumber
 \\
 =&\int_\beta^A H_A(\mu)
 \frac{Q_n(r^-_\mu)}{r^-_\mu\abs{u'(r^-_\mu)}^3}\dd\mu+\int_\beta^A H_A(\mu)
 \frac{Q_n(r^+_\mu)}{r^+_\mu\abs{u'(r^+_\mu)}^3}\dd\mu\nonumber\\=&
 \int_\beta^A H_A(\mu)\Psi(\mu)\dd\mu>0.
\end{align}

It remains to treat the middle interval $[b,\bar b]$.  If either $n\ge4$, or $n=3$ and $p+q\ge3$, Lemma~\ref{lem:Qn} gives $Q_n>0$ there, and hence the middle contribution is positive.  Suppose instead that $n=3$ and $p+q<3$.  If $\tau\le b$, then $Q_3\ge Q_1>0$ on $[b,\bar b]$.  If $b<\tau$, $Q_3>Q_1>0$ on $(\tau,\bar b]$, and
Lemma~\ref{lem:residual} gives
\[
 \int_b^\tau H_A(u)\frac{Q_3}{r u'^2}\dd r>0.
\]
Thus in every case
\begin{equation}\label{eq:middle-positive}
 \int_b^{\bar b}H_A(u)\frac{Q_n}{r u'^2}\dd r>0.
\end{equation}
Combining \eqref{eq:high-pair-positive} and \eqref{eq:middle-positive} in \eqref{eq:T2integral} gives
\[
 T_2(t^+)>T_2(t^-)>0.
\]
Thus the unique possible minimum of $T_2$ is positive, and $T_2(c_+)>0$.
\end{proof}

\begin{proposition}\label{prop:phase}
Assume \eqref{eq:ind-data}. Then
\begin{enumerate}[label=\textup{(\roman*)}]
\item $v$ has exactly one zero $\tau\in(c_-,r)$ on $[c_-,c_+]$.
\item $Q_1,Q_2>0$ on $[c_-,c_+]$.
\item $Q(c_+)>0$, $M(c_+)>0$, and $T_2(c_+)>0$.
\end{enumerate}
\end{proposition}

\begin{proof}
Lemmas~\ref{lem:prep} and~\ref{lem:river} give the unique zero of $v$ and the positivity of $Q_1,Q_2$ on $[c_-,c_+]$.  Lemmas~\ref{lem:river} and~\ref{lem:renewal} yield $Q(c_+)>0$ and $T_2(c_+)>0$.

It remains to recover $M(c_+)>0$.
At $\bar r$, $f=0$, $f'=p-q>0$, and $u'v>0$, hence $\rho(\bar r)>0$.  On $(\bar r,c_+)$, the signs of $f''(u)$ and $v$ agree, so \eqref{eq:rhoprime} gives $\rho'>0$.  Thus at $c_+$,
\[
 0<\rho(c_+)=-c_+^{n-1}f(u(c_+))v'(c_+),
\]
which forces $-u(c_+)v'(c_+)>0$ and therefore $M(c_+)>0$.
\end{proof}

\begin{corollary}
\label{addcor2}
Under the assumptions of Theorem~\ref{thm:firstphase}, $T_2(c_1)>0$.
\end{corollary}

\begin{proof}
Theorem~\ref{thm:firstphase} gives one zero $\tau_1$ of $v$,
$Q_1>0$ on $(0,c_1]$, and
\[
 Q>0\quad\text{on }(0,b_1],\qquad
  Q(\bar b_1)>Q(b_1)>0.
\]
The proof of Lemma~\ref{lem:phi} applies to the first phase with the initial positivity assumptions replaced by the conclusions of Theorem~\ref{thm:firstphase}.
Moreover,
since $M,T_1>0$ on $(0,b_1]$ and $g_1-g_2>0$ for $s>1$, we have
\[
 T_2=T_1+(g_1-g_2)M>0\quad\text{on }(0,b_1].
\]
For $n\ge4$, or $n=3$ with $p+q\ge3$,
Corollary~\ref{cor:Qn-first} provides the necessary pointwise
$Q_n$ positivity.  In the residual range $n=3$, $p+q<3$,
Corollary~\ref{addcor1} gives pointwise positivity if
$\tau_1\le b_1$, and the requisite weighted integral positivity if
$b_1<\tau_1$.  These are exactly the two alternatives used in the
middle-interval step of Lemma~\ref{lem:renewal}.  Applying the remaining
steps of that lemma to the first phase yields $T_2(c_1)>0$.
\end{proof}

Theorem~\ref{thm:firstphase} and Corollary~\ref{addcor2} give
\[
 Q(c_1)>0,\qquad M(c_1)>0,\qquad T_2(c_1)>0.
\]
Therefore Proposition~\ref{prop:phase} applies to every subsequent finite
nodal phase. Together with Lemma~\ref{lem:nodal-geometry}, we directly obtain

\begin{proposition}\label{cor:variation-finite}
Let $k\ge1$ and $u$ be a $k$-node bound state of \eqref{eq:shoot}.  If
\[
 0<z_1<\cdots<z_k
\]
are its zeros and
\[
 0=c_0<c_1<\cdots<c_k
\]
are the associated critical points, then $v=\partial_\alpha u$ has exactly one zero in $(0,z_1)$ and exactly one zero in each interval $(z_{i-1},z_i)$ for $2\le i\le k$.  At every nodal zero, one has
\[
 u'(z_i(\alpha),\alpha)v(z_i(\alpha),\alpha)>0,
 \qquad
 z_i'(\alpha)=-\frac{v(z_i(\alpha),\alpha)}{u'(z_i(\alpha),\alpha)}<0.
\]
Moreover
\[
 Q(c_k),\ M(c_k),\ T_2(c_k)>0.
\]
\end{proposition}

\section{Terminal tail and variation}\label{sec:tail}

We first normalize the terminal phase once and for all.  Let $u$ be a
bound state of \eqref{eq:shoot}, and let $c_*$ denote the initial point of
its terminal monotone branch.  Thus, for a nodal bound state, $c_*$ is the
unique critical point after the last zero.  By Lemma~\ref{lem:nodal-geometry}, together with the same
reflection used for the finite phases before \eqref{eq:phase-order}, we
may and shall assume throughout this section that
\begin{equation}\label{eq:terminal-normalization}
 u(c_*)>\beta,\qquad u'(c_*)=0,\qquad
 u(r)>0,\ u'(r)<0\quad \text{if } r>c_*,\qquad
 u(r)\rightarrow0\quad \text{as }r\to\infty.
\end{equation}
Whenever the variation $v=\partial_\alpha u$ is involved, the reflection is applied
simultaneously to $(u,v)$; as observed in Section~\ref{sec:transition},
all comparison functionals used below are unchanged.

Since $\beta>1$, the monotonicity in \eqref{eq:terminal-normalization}
uniquely determines two terminal levels
\begin{equation*}
{0<c_*<b_{*}<r_{*}<\infty},\qquad
 u(b_{*})=\beta,\qquad u(r_{*})=1.
\end{equation*}

\subsection{Terminal-tail estimate}

For the terminal separation in Section~\ref{sec:global}, the only
asymptotic information needed is the scale-invariant estimate
$r\abs{u'(r)}\to0$.  In the zero-mass problem, it follows immediately
from the radial flux.

\begin{proposition}\label{prop:terminal-tail}
Set
\[
 Z(r):=-r^{n-1}u'(r).
\]
Then $Z>0$ on $[r_{*},\infty)$ and $Z$ is strictly
decreasing on $(r_{*},\infty)$.  Moreover, for every $r\ge r_{*}$,
\begin{equation}\label{eq:terminal-tail-basic}
 0<-u'(r)\le Z(r_{*})r^{1-n},
 \qquad
 0<u(r)\le \frac{Z(r_{*})}{n-2}r^{2-n}.
\end{equation}
In particular, $r\abs{u'(r)}\rightarrow0$ as $r\to\infty$.
\end{proposition}

\begin{proof}
By \eqref{eq:terminal-normalization}, $Z>0$.  Since $0<u<1$ on
$(r_{*},\infty)$,
\[
 Z'(r)=r^{n-1}f(u(r))
 =-r^{n-1}u(r)^q\bigl(1-u(r)^{p-q}\bigr)<0.
\]
Hence $0<Z(r)\le Z(r_{*})$, which gives the first estimate in
\eqref{eq:terminal-tail-basic}.  Using $u(r)\to0$ and $u'<0$,
\[
 u(r)=\int_r^\infty -u'(s)\,\dd s
 \le Z(r_{*})\int_r^\infty s^{1-n}\,\dd s
 =\frac{Z(r_{*})}{n-2}r^{2-n}.
\]
The last assertion follows from
$0<-ru'(r)\le Z(r_{*})r^{2-n}\to0$.
\end{proof}




\subsection{Terminal variation}


The variation argument up to the level $u=\beta$ is exactly the local
part of the finite-phase argument in Lemma~\ref{lem:prep}.  We isolate the
result in the terminal notation introduced above.

\begin{lemma}\label{lem:tail-prep}
Assume that at $c_*$,
\begin{equation}\label{eq:tail-ind-data}
 Q(c_*)>0,\qquad M(c_*)>0,\qquad T_2(c_*)>0.
\end{equation}
Then $v$ has a unique zero $\tau_{*}$ on  $[c_*,s_{**}]$, where $s_{**}:=\max\{\tau_{*},b_{*}\}$. Moreover,
\begin{equation}\label{eq:tail-prep-conclusion}
 T_2,T_2',Q,M,Q_1,Q_2>0
 \quad\text{on }[c_*,s_{**}].
\end{equation}
\end{lemma}

\begin{proof}
Lemma~\ref{lem:Ptail} gives $P>0$ on the whole terminal branch, while
$u\ge1$ on $[c_*,r_{*}]$ gives $P_1>0$ there.
From \eqref{eq:tail-ind-data},
\[
 v(c_*)>0,\qquad v'(c_*)<0.
\]
If $v$ stayed positive up to $r_{*}$, then the same $\rho$-argument as in
Section~\ref{sec:transition} would give
$\rho(r_{*})>0$, whereas
\[
 \rho(r_{*})
 =r_{*}^{n-1}f'(1)u'(r_{*})v(r_{*})\le0.
\]
Thus the first zero $\tau_{*}$ lies in $(c_*,r_{*})$; moreover the
variational equation gives $v'<0$ on $(c_*,\tau_{*})$.

From this point, the proof of Lemma~\ref{lem:prep} is entirely local: it
uses only the interval from the initial critical point to the level
$u=\beta$, the positivity of $P$ and $P_1$, and the identities for
$Q,M,T_2,Q_1,Q_2$.  Applying that argument with
\[
 c_-=c_*,\qquad b=b_{*},\qquad r=r_{*}
\]
gives exactly \eqref{eq:tail-prep-conclusion} and excludes a second zero
of $v$ before $s_{**}$.
\end{proof}

\begin{proposition}\label{prop:tail-v}
The variation $v$ has exactly one zero on $(c_*,\infty)$, namely $\tau_{*}\in(c_*,r_{*})$.
Furthermore,
\begin{equation}\label{eq:terminal-transversality}
 v(r_{*})<0,\qquad
 B_0(r_{*})>0,
\end{equation}
and
\[
 Q_1>0\quad\text{on }(c_*,\infty).
\]
\end{proposition}

\begin{proof}
Proposition~\ref{cor:variation-finite} gives the terminal induction data
\eqref{eq:tail-ind-data}.  Lemma~\ref{lem:tail-prep} therefore yields the
first zero $\tau_{*}$ on $[c_*,s_{**}]$ and, with
$s_{**}=\max\{\tau_{*},b_{*}\}$, gives
\[
 v<0, \ Q_1>0
 \quad\text{immediately to the right of }s_{**},
 \qquad
 B_0(s_{**})=Q(s_{**})>0.
\]
Here the last identity follows because either $v(s_{**})=0$ or
$F(u(s_{**}))=0$.

It remains to continue these signs to infinity.  This is precisely the
one-sided bridge argument already used in Lemma~\ref{lem:river}.  Indeed,
as long as $v<0$ and $Q_1>0$ after $s_{**}$, one has
$u'v>0$ and hence $Q_n>Q_1>0$; since $F(u)<0$ there,
\eqref{eq:B0prime} gives $B_0'>0$.  A second zero of $v$ cannot be the
first bad event, because at such a zero
\[
 Q_1=Q=r^nu'v'<0.
\]
If instead $Q_1$ vanished first while $v<0$, then
\[
 B_0=-2r^{n-1}\frac{v}{u'}E<0,
\]
contradicting $B_0>B_0(s_{**})>0$.  Hence no bad event occurs.  Thus
$\tau_{*}$ is the unique zero of $v$ on $(c_*,\infty)$ and
$Q_1>0$ on $(s_{**},\infty)$. This together with \eqref{eq:tail-prep-conclusion} yields $Q_1>0$ on $(c_{*},\infty)$. 

Finally, $s_{**}<r_{*}$ gives $v(r_{*})<0$, while the same strict monotonicity of $B_0$ gives
$B_0(r_{*})>0$.  This proves \eqref{eq:terminal-transversality}.
\end{proof}

\section{Finite-level separation and global uniqueness}\label{sec:global}

Tang closes the massive problem by studying the zero-count function of the
shooting family.  In particular, the jumping step in
\cite[Lemma~3.1]{Tang2026} uses the terminal information that
the variation is eventually strictly monotone and satisfies
\(\abs{v(r)}\to\infty\).  That mechanism is naturally adapted to the
massive tail, where \(f'(0)<0\).  In the present zero-mass problem,
\(f'(0)=0\), and Proposition~\ref{prop:tail-v} gives the terminal
transversality needed here, but not such a divergence statement for \(v\).
A direct adaptation of that final step would therefore require additional
asymptotic information for the zero-mass linearized equation.

Instead, we close the zero-mass problem at a finite regular level, following
the shooting-topological strategy 
of
Cort\'azar--Garc\'ia-Huidobro--Yarur \cite{CGHY2011}.  The crucial point is
that the terminal transversality obtained in Proposition~\ref{prop:tail-v}
already yields the strict information needed at the regular level
\(\abs{u}=1\).  Indeed, after reflecting the terminal branch when necessary,
the inequalities
\[
 \frac{v}{u'}>0,
 \qquad B_0>0
\]
at this level imply, through Lemma~\ref{lem:Wderivative}, that the inverse
radius decreases with the shooting parameter while the separation quantity
\(W\) increases.  Thus nearby terminal inverse branches acquire a strict
ordering at the fixed level \(\abs{u}=1\).  Lemma~\ref{lem:terminal-separation}
then propagates this ordering down to the lower endpoint of the terminal
branch, where it distinguishes the crossing, decaying, and turning
alternatives and yields the local shooting separation property without requiring any
asymptotic divergence of the variation as \(r\to\infty\).  For the subsequent shooting-topological argument, we use only the basic assumptions $(f_1)$--$(f_2)$ of \cite{CGHY2011}, the openness and local persistence properties established in \cite[Propositions~2.2--2.3]{CGHY2011}.
In particular, the
additional global comparison conditions imposed in \cite{CGHY2011} are not
invoked here; their role is replaced by the finite-level transversality and
separation mechanism established above.

We first verify that the double-power nonlinearity satisfies the hypotheses of the prescribed-node existence theorem and the basic assumptions required for the shooting results used below.

\begin{lemma}\label{lem:external-hyp}
Under assumption \eqref{eq:main-range}, the prescribed-node existence theorem
\cite[Theorem~1]{MTW1990} applies to \(f\).  Moreover, \(f\) satisfies the
basic assumptions $(f_1)$--$(f_2)$ of \cite{CGHY2011}.
\end{lemma}

\begin{proof}
The function $f(s)=\abs{s}^{p-1}s-\abs{s}^{q-1}s$
is odd, belongs to \(C^1(\R)\), and is locally Lipschitz.  On \((0,\infty)\),
\[
 f<0\quad\hbox{on }(0,1),\qquad f(1)=0,
 \qquad
 f>0\quad\hbox{on }(1,\infty).
\]
Its primitive $F(s)$
has the unique positive zero $\beta>1$,
and \(F(s)\to\infty\) as \(\abs{s}\to\infty\).  Moreover,
\[
 f(s)=\abs{s}^{p-1}s+o(\abs{s}^{p})
 \quad\text{as }\abs{s}\to\infty,
\]
and \(1<p<(n+2)/(n-2)\).
These are the hypotheses of \cite[Theorem~1]{MTW1990}
for the present double-power nonlinearity, so for
every prescribed finite number of zeros, there is a radial solution tending
to zero at infinity.  Every finite zero is simple by ODE uniqueness.

For \cite{CGHY2011}, take \(b=1\).  The sign properties above and the zero
\(\beta\) of \(F\) give $(f_1)$.  Furthermore,
\[
 \abs{f'(s)}\le p s^{p-1}+q s^{q-1}\in L^1(0,1),
\]
so $(f_2)$ also holds.
\end{proof}

\subsection{Finite-level terminal separation}

Set 
\begin{equation*}
E(r,\alpha):=\frac12 u'(r,\alpha)^2+F(u(r,\alpha)).
\end{equation*}
The next
identity is the point at which the phase analysis of Section~\ref{sec:first}--\ref{sec:tail} enters the
finite-level comparison.

\begin{lemma}\label{lem:Wderivative}
Let \(s\) be a fixed regular level on a monotone branch of the shooting
family, and 
denote by $R(\alpha):=r(s,\alpha)$ the corresponding inverse radius. Then
\[
 u(R(\alpha),\alpha)=s,
 \qquad
 u'(R(\alpha),\alpha)\neq0.
\]
Moreover, \(R\) is \(C^1\) and
\begin{equation}\label{eq:Ralpha}
 R'(\alpha)
 =-\frac{v(R(\alpha),\alpha)}
        {u'(R(\alpha),\alpha)}.
\end{equation}
Set
\[
 D(s,\alpha)
 :=u'(R(\alpha),\alpha)^2+2F(s)
 =2E(R(\alpha),\alpha).
\]
Then the quantity
\[
 W(s,\alpha)
 :=R(\alpha)\sqrt{D(s,\alpha)}\quad \text{if }D(s,\alpha)>0
\]
satisfies
\begin{equation}\label{eq:Walpha}
 \partial_\alpha W(s,\alpha)
 =
 \frac{B_0(R(\alpha),\alpha)}
 {R(\alpha)^{n-1}\sqrt{D(s,\alpha)}}.
\end{equation}
\end{lemma}

\begin{proof}
Since \(s\) is a regular value on the chosen branch,
\(u'(R(\alpha),\alpha)\neq0\).  The implicit function theorem, together
with the \(C^1\)-dependence of the shooting solution on \(\alpha\), shows
that \(R(\alpha)\) is \(C^1\).  Differentiating $u(R(\alpha),\alpha)=s$
with respect to \(\alpha\) gives
\[
 u'(R(\alpha),\alpha)R'(\alpha)
 +v(R(\alpha),\alpha)=0,
\]
which proves \eqref{eq:Ralpha}.

We next differentiate \(W\) with \(s\) fixed.  To simplify the notation,
in the remainder of the proof all quantities \(u,u',u'',v,v'\) are
evaluated at \((R(\alpha),\alpha)\), and we write \(R=R(\alpha)\) and
\(D=D(s,\alpha)\).  Then 
\[
 \partial_\alpha u'
 =v'+u''R'
 =v'-u''\frac{v}{u'}.
\]
Since \(F(s)\) is independent of \(\alpha\),
\[
 \partial_\alpha D
 =2u'\left(
 v'-u''\frac{v}{u'}
 \right).
\]
Therefore
\begin{align*}
 \sqrt D\partial_\alpha W
 =
 R'D+\frac{R}{2}\partial_\alpha D
 =
 -\frac{v}{u'}D
 +R\left(u'v'-u''v\right).
\end{align*}
Using
\[
 u''
 =-\frac{n-1}{R}u'-f(u),
 \qquad
 u=s,
\]
we obtain
\begin{align*}
 \sqrt D\partial_\alpha W
 =&
 -\frac{v}{u'}\bigl(u'^2+2F(u)\bigr)
 +R u'v'+(n-1)u'v+R f(u)v
 \\=&
 R\bigl(u'v'+f(u)v\bigr)
 +(n-2)u'v
 -2F(u)\frac{v}{u'}\\
 =&\frac{1}{R^{n-1}}
 \left(
 Q-2F(u)\frac{R^{n-1}v}{u'}
 \right)=\frac{B_0(R,\alpha)}{R^{n-1}},
\end{align*}
and \eqref{eq:Walpha} follows.  
\end{proof}

We now use the shooting notation of \cite{CGHY2011}.  For \(\alpha>1\), let
\[
 Z_1(\alpha):=\sup\{r>0:u(s,\alpha)>0\ \text{and}\
 u'(s,\alpha)<0\ \text{for all }s\in(0,r)\}.
\]
Set
\begin{align*}
 \mathcal N_1&:=\{\alpha>1:u(Z_1(\alpha),\alpha)=0,
                    \ u'(Z_1(\alpha),\alpha)<0\},\\
 \mathcal G_1&:=\{\alpha>1:u(Z_1(\alpha),\alpha)=0,
                    \ u'(Z_1(\alpha),\alpha)=0\},\\
 \mathcal P_1&:=\{\alpha>1:u(Z_1(\alpha),\alpha)>0\}.
\end{align*}
Values at \(Z_1(\alpha)=\infty\) are understood as limits.  Inductively, for
\(m\ge2\), define
\[
 \widetilde{\mathcal F}_m
 :=\{\alpha\in\mathcal N_{m-1}:(-1)^m u'(r,\alpha)<0
       \ \text{for every }r>Z_{m-1}(\alpha)\}.
\]
For \(\alpha\in\mathcal N_{m-1}\setminus\widetilde{\mathcal F}_m\), put
\begin{align*}
 C_{m-1}(\alpha)
 &:=\inf\{r>Z_{m-1}(\alpha):u'(r,\alpha)=0\},\\
 Z_m(\alpha)
 &:=\sup\{r>C_{m-1}(\alpha):(-1)^m u(s,\alpha)<0 \ \text{and} \
        (-1)^m u'(s,\alpha)>0 \ \text{for all }s\in(C_{m-1}(\alpha),r)\},
\end{align*}
and if \(\alpha\in\widetilde{\mathcal F}_m\), we set $C_{m-1}(\alpha)=\infty$.
Then set
\begin{align*}
 \mathcal F_m
 &:=\{\alpha\in\mathcal N_{m-1}\setminus\widetilde{\mathcal F}_m:
       (-1)^m u(Z_m(\alpha),\alpha)<0\},\\
 \mathcal N_m
 &:=\{\alpha\in\mathcal N_{m-1}\setminus\widetilde{\mathcal F}_m:
       u(Z_m(\alpha),\alpha)=0,
       \ (-1)^m u'(Z_m(\alpha),\alpha)>0\},\\
 \mathcal G_m
 &:=\{\alpha\in\mathcal N_{m-1}\setminus\widetilde{\mathcal F}_m:
       u(Z_m(\alpha),\alpha)=u'(Z_m(\alpha),\alpha)=0\},\\
 \mathcal P_m&:=\widetilde{\mathcal F}_m\cup\mathcal F_m.
\end{align*}
Values at \(Z_m(\alpha)=\infty\) are understood as limits.
A finite double zero is impossible by ODE uniqueness; thus \(\mathcal G_m\)
is precisely the class of $(m-1)$-bound states,
\(\mathcal N_m\) is the next-crossing class, and \(\mathcal P_m\) contains
the two noncrossing alternatives.

\begin{lemma}\label{lem:terminal-separation}
Let $m\ge2$ and $\alpha_*\in\mathcal G_m$.
Let $\alpha_1<\alpha_2$ be sufficiently close to
$\alpha_*$, with
$\alpha_j\in\mathcal N_{m-1}
\setminus\widetilde{\mathcal F}_m$,
$j=1,2$.
Put
\[
 w_j(r):=(-1)^{m-1} u(r,\alpha_j)
\]
on the terminal branch following \(C_{m-1}(\alpha_j)\).  Then
\(w_j>0\) and \(w_j'<0\).  Let \(r_j=r_j(s)\) be the inverse of this
branch and set
\[
 y_j(s):=-w_j'(r_j(s))>0,
 \qquad
 D_j(s):=y_j(s)^2+2F(s),
 \qquad
 W_j(s):=r_j(s)\sqrt{D_j(s)}\quad \text{if }D_j(s)>0.
\]
Let
\[
 L_j:=\lim_{r\uparrow Z_m(\alpha_j)}w_j(r).
\]
Then
\begin{equation}\label{eq:Lj-beta}
L_j=
\begin{cases}
0,
\qquad& \alpha_j\in\mathcal G_m\cup\mathcal N_m,\\
\in(0,1),
\qquad& \alpha_j\in\mathcal F_m,\quad Z_m(\alpha_j)<\infty,\\
1,
\qquad& \alpha_j\in\mathcal F_m,\quad Z_m(\alpha_j)=\infty.
\end{cases}
\end{equation}
Suppose that there exists $U\in (\max\{L_1,L_2\},\beta)$ such that 
$D_j(U)>0$,
and
\begin{equation}\label{eq:terminal-initial-order}
 r_1(U)\ge r_2(U),
 \qquad
 W_1(U)<W_2(U).
\end{equation}
Define
\[
 S_j:=\inf\{s\in(L_j,U]:D_j(s)>0\}.
\]
Then $S_1\ge S_2\geq 0$,
and for every \(s\in(S_1,U)\),
\begin{equation*}
 r_1(s)>r_2(s),
 \qquad
 W_1(s)<W_2(s),
 \qquad
 y_1(s)<y_2(s).
\end{equation*}
Moreover,
\begin{equation*}
 W_1(s)-W_2(s)
 \le W_1(U)-W_2(U)<0,
 \qquad s\in(S_1,U],
\end{equation*}
and
\begin{equation*}
 S_j=0
 \quad\Longleftrightarrow\quad
 \alpha_j\in\mathcal G_m\cup\mathcal N_m .
\end{equation*}
\end{lemma}
\begin{proof}
For $L_j$, if
\(\alpha_j\in\mathcal G_m\cup\mathcal N_m\), the definitions of these
sets immediately give $L_j=0$. Suppose next that \(\alpha_j\in\mathcal F_m\) and
\(Z_m(\alpha_j)<\infty\).  By the definition of \(\mathcal F_m\), $L_j>0$.
Since \(Z_m(\alpha_j)\) is the maximal right endpoint on which $w_j>0$, $w_j'<0$,
and \(w_j(Z_m(\alpha_j))=L_j>0\), we must have $w_j'(Z_m(\alpha_j))=0$.  Moreover, $w_j''(Z_m(\alpha_j))>0$.
As \(f\) is odd, \(w_j\) satisfies
\begin{equation}\label{add6.31}
w_j''+\frac{n-1}{r}w_j'+f(w_j)=0.
\end{equation}
Evaluating \eqref{add6.31} at \(r=Z_m(\alpha_j)\) therefore yields $0<
w_j''(Z_m(\alpha_j))=-f(L_j)$,
and hence $L_j\in (0,1)$.
It remains to consider
\(\alpha_j\in\mathcal F_m\) with \(Z_m(\alpha_j)=\infty\).
In this case,
\begin{equation*}
w_j(r)\rightarrow L_j>0\quad \text{as } r\to\infty.
\end{equation*}
Set
\[
E_j(r):=\frac12 w_j'(r)^2+F(w_j(r)).
\]
Since \(w_j\) satisfies the same radial equation as \(u\), $E_j'=-\frac{n-1}{r}w_j'^2\le0$.
As \(F\) is bounded from below, \(E_j(r)\) has a finite limit as
\(r\to\infty\).  Since \(w_j\to L_j\), it follows 
that \(w_j'^2\) also has a finite limit.  This limit must be zero;
otherwise there would exist \(c>0\) and \(R>0\) such that $-w_j'(r)\ge c$ for every $r\ge R$,
which is incompatible with the convergence of \(w_j\) to the finite
limit \(L_j\).  Hence
\[
w_j'(r)\rightarrow0
\quad\text{as }r\to\infty.
\]
Using again the equation \eqref{add6.31},
we obtain
\[
w_j''(r)\rightarrow-f(L_j)\quad\text{as }r\to\infty.
\]
If \(f(L_j)\ne0\), then \(w_j''\) is eventually bounded away from zero
and has a fixed sign, which contradicts
\(w_j'\to0\).  Thus $f(L_j)=0$.
Since \(L_j>0\), we conclude
that $L_j=1$. This proves \eqref{eq:Lj-beta}.

Since \(w_j\) satisfies \eqref{add6.31}, on each inverse branch,
\begin{equation*}
  r_j'(s)=-\frac1{y_j(s)},
 \qquad
 y_j'(s)=\frac{n-1}{r_j(s)}-\frac{f(s)}{y_j(s)}.
\end{equation*}
Consequently
\begin{equation}\label{eq:terminal-Dprime}
 D_j'(s)=\frac{2(n-1)y_j(s)}{r_j(s)}>0,
\end{equation}
and
\[
 W_j'(s)=H_s(y_j(s)),
 \qquad H_s(y):=
 -\frac{\sqrt{y^2+2F(s)}}{y}
 +(n-1)\frac{y}{\sqrt{y^2+2F(s)}}.
\]
For $y^2+2F(s)>0$, since  \(F(s)<0\) for \(s\in(0,\beta)\), 
\begin{equation*}
  \partial_y H_s(y)
 =
 \frac{2F(s)}
 {y^2\sqrt{y^2+2F(s)}}
 +
 \frac{2(n-1)F(s)}
 {(y^2+2F(s))^{\frac{3}{2}}}
 <0.
\end{equation*}

The functional $W$ is the inverse-branch quantity used in
\cite[Lemma~4.3]{CGHY2011}. For completeness, we give
a direct proof of the separation result needed here,
valid for every $n\ge3$.

By \eqref{eq:terminal-Dprime}, $D_j$ is strictly increasing
on $(L_j,U]$. If
$\alpha_j\in\mathcal G_m\cup\mathcal N_m$, then $L_j=0$ and
\[
\lim_{s\downarrow0}D_j(s)=
\begin{cases}
0, \qquad& \alpha_j\in\mathcal G_m,\\
u'(Z_m(\alpha_j),\alpha_j)^2>0,
   \qquad& \alpha_j\in\mathcal N_m.
\end{cases}
\]
The first limit follows from
Proposition~\ref{prop:terminal-tail}.
Hence $D_j(s)>0$ for $s\in(0,U]$, and $S_j=0$.
If $\alpha_j\in\mathcal F_m$, then $L_j\in (0,\beta)$
and
\[
\lim_{s\downarrow L_j}D_j(s)=2F(L_j)<0.
\]
Thus $D_j$ has a unique zero $S_j\in(L_j,U)$. This proves the asserted characterization of \(S_j\).

The initial ordering \eqref{eq:terminal-initial-order}
gives
\[
\frac{W_1(U)}{r_1(U)}
<
\frac{W_2(U)}{r_2(U)},
\]
and hence $D_1(U)<D_2(U)$ and $y_1(U)<y_2(U)$.
More generally, whenever $r_1\ge r_2$ and
$W_1<W_2$ on the common positivity interval, we have
$D_1<D_2$ and $y_1<y_2$. Therefore,
\[
(r_1-r_2)'=-\frac1{y_1}+\frac1{y_2}<0,
\qquad
(W_1-W_2)'=H_s(y_1)-H_s(y_2)>0,
\]
where the second inequality follows from
$\partial_yH_s<0$. Integrating backward from $U$, we find that,
as long as the above orderings hold,
\[
\begin{aligned}
r_1(s)-r_2(s)
=r_1(U)-r_2(U)
  -\int_s^U(r_1-r_2)'(t)\,\dd t>r_1(U)-r_2(U)\ge0,
\end{aligned}
\]
and
\[
\begin{aligned}
W_1(s)-W_2(s)
=W_1(U)-W_2(U)
  -\int_s^U(W_1-W_2)'(t)\,\dd t<W_1(U)-W_2(U)<0.
\end{aligned}
\]
These inequalities rule out any first contact
as $s$ decreases from $U$.
Consequently,
\[
r_1(s)>r_2(s),\qquad
W_1(s)<W_2(s),\qquad
y_1(s)<y_2(s)
\]
throughout the common positivity interval
$(\max\{S_1,S_2\},U)$.


It remains to compare the lower endpoints.
Suppose that $S_1<S_2$. Then $S_2>0$ and $D_2(S_2)=0$, hence $\alpha_2\in\mathcal F_m$. By the preceding characterization, $S_2\in(L_2,U)$.
Since $r_2(S_2)<\infty$, we have $\lim_{s\downarrow S_2}W_2(s)=0$.
On the other hand, $D_1(S_2)>0$ and
$r_1(S_2)<\infty$, so $W_1(S_2)>0$.
This contradicts $W_1<W_2$ on $(S_2,U)$.
Thus $S_1\ge S_2$.
\end{proof}

\subsection{Local shooting separation property}

\begin{proposition}\label{prop:key}
Let \(k\ge1\) and \(\alpha_*\) be the shooting height
of a \(k\)-node bound state to \eqref{eq:shoot}.  Then there
exists \(\delta>0\) such that
\[
 (\alpha_* -\delta,\alpha_*)\subset\mathcal P_{k+1},
 \qquad
 {(\alpha_*,\alpha_*+\delta)
 \subset\mathcal N_{k+1}.}
\]
\end{proposition}

\begin{proof}
Put \(m:=k+1\), then
\(\alpha_*\in\mathcal G_m\). Lemma~\ref{lem:external-hyp} and
\cite[Proposition~2.3]{CGHY2011} imply, after shrinking \(\delta>0\), that
\[
 (\alpha_*-\delta,\alpha_*+\delta)
 \subset \mathcal N_{m-1}\setminus\widetilde{\mathcal F}_m .
\]
Thus the terminal branch persists for all such \(\alpha\).

Set
\[
 w(r,\alpha):=(-1)^k u(r,\alpha),\qquad
 \psi(r,\alpha):=(-1)^k v(r,\alpha).
\]
Then \(w>0\) and \(w'<0\). 
For the reference profile, the level \(w=1\) is regular.  Hence, after
reducing \(\delta\) if necessary, the implicit function theorem gives a
nearby continuation \(R(\alpha)\) such that
\[
 w(R(\alpha),\alpha)=1,
 \qquad
 w'(R(\alpha),\alpha)<0.
\]
For any
\(r_0\in(R(\alpha),Z_m(\alpha))\), we have
\[
 \lim_{r\uparrow Z_m(\alpha)}w(r,\alpha)
 \le w(r_0,\alpha)
 <w(R(\alpha),\alpha)=1.
\]
At $\alpha=\alpha_*$, Proposition~\ref{prop:tail-v} and the
simultaneous reflection $(u,v)\mapsto(w,\psi)$ give $\psi(R(\alpha_*),\alpha_*)<0$ and $B_0(R(\alpha_*),\alpha_*)>0$.
Since $R(\alpha)$,
$\psi(R(\alpha),\alpha)$ depend continuously on $\alpha$,
and $B_0$ is continuous where the radial derivative does not vanish,
these strict inequalities persist for $\alpha$ in a sufficiently
small neighborhood of $\alpha_*$.  Thus
\[
 \psi(R(\alpha),\alpha)<0,
 \qquad B_0(R(\alpha),\alpha)>0.
\]
Since
\[
 w'(R(\alpha),\alpha)^2+2F(1)
 =2E(R(\alpha),\alpha)>0,
\]
continuity and Lemma~\ref{lem:Wderivative} give, after reducing \(\delta\),
\[
 R'(\alpha)<0,\qquad
 \partial_\alpha W(1,\alpha)>0.
\]
Consequently, for every \(\alpha_1<\alpha_2\) sufficiently close to \(\alpha_*\),  
\[
 r_1(1)>r_2(1),\qquad W_1(1)<W_2(1),
\]
and Lemma~\ref{lem:terminal-separation} applies with \(U=1\).

First let \(\alpha_1=\alpha_*<\alpha_2\). Since
\(\alpha_1\in\mathcal G_m\), by Lemma~\ref{lem:terminal-separation}, we have \(S_1=0\), hence 
\(S_2=0\), and therefore $\alpha_2\in\mathcal G_m\cup\mathcal N_m$.
If \(\alpha_2\in\mathcal G_m\), both shots are decaying. By
Proposition~\ref{prop:terminal-tail}, $r_j(s)y_j(s)\to0$ as $s\rightarrow0$,
and, since \(F(s)<0\),
\[
 0<W_j(s)\le r_j(s)y_j(s)\to0.
\]
This contradicts
\[
 W_1(s)-W_2(s)
 \le W_1(1)-W_2(1)<0.
\]
Thus \(\alpha_2\in\mathcal N_m\).

Now let \(\alpha_1<\alpha_2=\alpha_*\). Then  \(\alpha_2\in\mathcal G_m\) and \(S_2=0\). If \(S_1=0\), then
\(\alpha_1\in\mathcal G_m\cup\mathcal N_m\). The case
\(\alpha_1\in\mathcal G_m\) is excluded by the same \(W\)-gap argument.
If \(\alpha_1\in\mathcal N_m\), Lemma~\ref{lem:terminal-separation} gives
\(y_1(s)<y_2(s)\) for small $s>0$, whereas
\[
 \lim_{s\downarrow0}y_1(s)
 =\abs{u'(Z_m(\alpha_1),\alpha_1)}>0,
 \qquad
 \lim_{s\downarrow0}y_2(s)=\abs{u'(Z_m(\alpha_2),\alpha_2)}=0,
\]
a contradiction. Hence \(S_1>0\), and therefore $\alpha_1\in\mathcal F_m\subset\mathcal P_m$. This completes the proof.
\end{proof}

\subsection{Global shooting topology}

We now use only the openness of  \(\mathcal P_j\) and \(\mathcal N_j\) from
\cite[Proposition~2.2]{CGHY2011} and the local shooting separation property just proved.
Inside \(\mathcal N_{j-1}\), the definitions give the disjoint decomposition
\[
 \mathcal N_{j-1}
 =\mathcal P_j\mathbin{\dot\cup}\mathcal G_j\mathbin{\dot\cup}\mathcal N_j.
\]
Thus the remaining argument is purely one-dimensional topology in the
shooting parameter. The following standard fact is used only to initialize the global
shooting induction. 

\begin{lemma}\label{lem:ground-threshold}
Assume \eqref{eq:main-range}. Then there exists a unique $\alpha_0>1$ such that
\[
\mathcal G_1=\{\alpha_0\}.
\]
Moreover,
\[
\mathcal P_1=(1,\alpha_0),
\qquad
\mathcal N_1=(\alpha_0,\infty).
\]
\end{lemma}

\begin{proof}
Existence follows from Lemma~\ref{lem:external-hyp}, while uniqueness
follows from the Serrin--Tang criterion
\cite[Theorem~1]{SerrinTang2000}, since for $s>1$,
\[
K(s):=\frac{s f'(s)}{f(s)}
=\frac{p s^{d}-q}{s^{d}-1},
\qquad
K'(s)
=-\frac{(d)^2s^{d-1}}
       {(s^{d}-1)^2}<0.
\]
Thus $\mathcal G_1=\{\alpha_0\}$.

Since \(E(r,\alpha_0)\to0\) as \(r\to\infty\) and \(E\) is strictly decreasing, we have \(E(r,\alpha_0)>0\) for every finite \(r\).
Consequently
\(F(\alpha_0)=E(0)>0\), so \(\alpha_0>\beta\). On the other hand, for the shot with
\(\alpha=\beta\), one has  $E(0,\beta)=F(\beta)=0$. 
Since \(E\) is decreasing,
this shot cannot reach \(u=0\), and therefore
\(\beta\in\mathcal P_1\).

By \cite[Proposition~2.2]{CGHY2011}, the sets
\(\mathcal P_1\) and \(\mathcal N_1\) are open and $(1,\infty)
 =\mathcal P_1\mathbin{\dot\cup}
  \mathcal G_1\mathbin{\dot\cup}\mathcal N_1$.
Since \(\beta\in(1,\alpha_0)\cap\mathcal P_1\), connectedness yields $\mathcal P_1=(1,\alpha_0)$.
Moreover, Lemma~\ref{lem:external-hyp} provides a \(1\)-node bound state, so \(\mathcal N_1\ne\varnothing\). Connectedness then gives $\mathcal N_1=(\alpha_0,\infty)$.
\end{proof}

\begin{lemma}\label{lem:jumping}
Assume that $\mathcal N_{j-1}=
\mathcal P_j\mathbin{\dot\cup}
\mathcal G_j\mathbin{\dot\cup}
\mathcal N_j$, $j\ge2$,
where $\mathcal P_j$ and $\mathcal N_j$ are relatively open in
$\mathcal N_{j-1}$.
If $\mathcal N_{j-1}=(\alpha_{j-2},\infty)$
and $\mathcal G_j\neq\varnothing$, then there exists a unique
$\alpha_{j-1}>\alpha_{j-2}$ such that
\[
\mathcal G_j=\{\alpha_{j-1}\},
\qquad
\mathcal P_j=(\alpha_{j-2},\alpha_{j-1}),
\qquad
\mathcal N_j=(\alpha_{j-1},\infty).
\]
\end{lemma}

\begin{proof}
Take $\gamma\in\mathcal G_j$. By Proposition~\ref{prop:key},
$\mathcal N_j$ contains a right neighborhood of $\gamma$.
We claim that $(\gamma,\infty)\subset\mathcal N_j$.
Otherwise, let
\[
b=\sup\{s>\gamma:(\gamma,s)\subset\mathcal N_j\}<\infty.
\]
Since $\mathcal N_j$ is relatively open, $b\notin\mathcal N_j$.
Also $b\notin\mathcal P_j$, because the openness of $\mathcal P_j$
would contradict $(\gamma,b)\subset\mathcal N_j$.
Hence $b\in\mathcal G_j$. The local shooting separation property at $b$ gives $(b-\delta,b)\subset\mathcal P_j$
for some $\delta>0$, again contradicting
$(\gamma,b)\subset\mathcal N_j$. Thus
\[
(\gamma,\infty)\subset\mathcal N_j.
\]
This immediately implies that $\gamma$ is the only point of
$\mathcal G_j$: if $\eta>\gamma$ belonged to $\mathcal G_j$, then
$\eta\in\mathcal N_j$; while if $\eta<\gamma$ belonged to
$\mathcal G_j$, applying the preceding argument to $\eta$ would give
$\gamma\in\mathcal N_j$.

Therefore
\[
\mathcal G_j=\{\gamma\}.
\]
Since
\[
(\alpha_{j-2},\gamma)
=
\bigl(\mathcal P_j\cap(\alpha_{j-2},\gamma)\bigr)
\mathbin{\dot\cup}
\bigl(\mathcal N_j\cap(\alpha_{j-2},\gamma)\bigr),
\]
and both sets are relatively open, the connectedness of
$(\alpha_{j-2},\gamma)$, together with the left-hand local shooting separation property,
yields
\[
(\alpha_{j-2},\gamma)\subset\mathcal P_j.
\]
Setting $\alpha_{j-1}=\gamma$ completes the proof.
\end{proof}


\begin{proof}[Proof of Theorem~\ref{thm:main}]
Lemma~\ref{lem:external-hyp} gives $\mathcal G_{k+1}\neq\varnothing$
for every \(k\ge1\), since a profile in \(\mathcal G_{k+1}\) has
exactly \(k\) zeros.
For \(k\ge1\), Proposition~\ref{prop:tail-v} provides the terminal
transversality
\[
v(r_*)<0,\qquad B_0(r_*)>0.
\]
Here Proposition~\ref{prop:tail-v} relies on the finite-phase induction
summarized in Proposition~\ref{cor:variation-finite}, which in turn
follows from Theorem~\ref{thm:firstphase},
Corollary~\ref{addcor2}, and Proposition~\ref{prop:phase}.
Consequently, Proposition~\ref{prop:key} yields the local shooting separation property
at every point of \(\mathcal G_j\), \(j\ge2\).

Lemma~\ref{lem:ground-threshold} gives the initial level \(j=1\).
Applying Lemma~\ref{lem:jumping} inductively, we conclude that $\mathcal G_{k+1}$
consists of exactly one shooting height for every \(k\ge1\).
Since the radial initial-value problem is uniquely determined by the
shooting height \(\alpha=u(0)>0\), there is at most one radial bound
state with exactly \(k\) nodes and positive value at the origin.
Together with the existence supplied by
Lemma~\ref{lem:external-hyp}, this proves Theorem~\ref{thm:main}.
\end{proof}

\begin{corollary}\label{cor:shooting-structure}
Let $\alpha_k=u_k(0)$ for $k\ge1$, and let
$\alpha_0$ denote the ground-state shooting height.
Then
\[
\beta<\alpha_0<\alpha_1<\alpha_2<\cdots.
\]
Moreover, the shooting variation
$v=\partial_\alpha u$ along the $k$-node
bound state has exactly $k+1$ simple zeros
in $(0,\infty)$.
\end{corollary}

\begin{proof}
The strict ordering follows from
Lemmas~\ref{lem:ground-threshold}
and~\ref{lem:jumping}.
The exact zero count follows from
Propositions~\ref{cor:variation-finite}
and \ref{prop:tail-v}.
All zeros are simple by uniqueness
for the variational equation.
\end{proof}


\vspace{.3cm}

\noindent\textbf{Conflict of interest.}
The authors declare that they have no conflict of interest.

\medskip
\noindent\textbf{Data availability.}
Data sharing is not applicable to this article as no data were created or analyzed in this study.


\end{document}